\pdfoutput=1
\RequirePackage{ifpdf}
\ifpdf % We are running pdfTeX in pdf mode
\documentclass[pdftex]{sigma}
\else
\documentclass{sigma}
\fi

\usepackage{bm}

\newcommand{\Res}{\mathop{\rm Res}}

\newcommand {\pp}{\bm{p}}
\newcommand{\zz}{\bm{z}}
\newcommand{\ci}{\mathcal{I}}
\newcommand{\cl}{\mathcal{L}}
\newcommand{\bn}{\mathbb{N}}
\newcommand{\bp}{\mathbb{P}}
\newcommand{\bq}{\mathbb{Q}}
\newcommand{\bz}{\mathbb{Z}}
\newcommand{\dd}{\mathrm{d}}
\newcommand {\h}{\hbar}
\newcommand{\modm}{\mathcal M}

\begin{document}
\allowdisplaybreaks

\newcommand{\arXivNumber}{1905.01890}

\renewcommand{\thefootnote}{}

\renewcommand{\PaperNumber}{061}

\FirstPageHeading

\ShortArticleName{Loop Equations for Gromov--Witten Invariant of $\mathbb{P}^1$}

\ArticleName{Loop Equations for Gromov--Witten Invariant of $\boldsymbol{\mathbb{P}^1}$\footnote{This paper is a~contribution to the Special Issue on Integrability, Geometry, Moduli in honor of Motohico Mulase for his 65th birthday. The full collection is available at \href{https://www.emis.de/journals/SIGMA/Mulase.html}{https://www.emis.de/journals/SIGMA/Mulase.html}}}

\Author{Ga\"etan BOROT~$^\dag$ and Paul NORBURY~$^\ddag$}

\AuthorNameForHeading{G.~Borot and P.~Norbury}

\Address{$^\dag$~Max Planck Institut f\"ur Mathematik, Vivatsgasse 7, 53111 Bonn, Germany}
\EmailD{\href{mailto:gborot@mpim-bonn.mpg.de}{gborot@mpim-bonn.mpg.de}}

\Address{$^\ddag$~School of Mathematics and Statistics, University of Melbourne, VIC 3010, Australia}
\EmailD{\href{mailto:norbury@unimelb.edu.au}{norbury@unimelb.edu.au}}

\ArticleDates{Received May 16, 2019, in final form August 14, 2019; Published online August 23, 2019}

\Abstract{We show that non-stationary Gromov--Witten invariants of $\bp^1$ can be extracted from open periods of the Eynard--Orantin topological recursion correlators $\omega_{g,n}$ whose Laurent series expansion at $\infty$ compute the stationary invariants. To do so, we overcome the technical difficulties to global loop equations for the spectral $x(z) = z + 1/z$ and $y(z) = \ln z$ from the local loop equations satisfied by the $\omega_{g,n}$, and check these global loop equations are equivalent to the Virasoro constraints that are known to govern the full Gromov--Witten theory of $\bp^1$.}


\Classification{32G15; 14D23; 53D45}

\renewcommand{\thefootnote}{\arabic{footnote}}
\setcounter{footnote}{0}

\section{Introduction}\label{S1}
The Gromov--Witten invariants of $\bp^1$ were expressed as expectation values of Plancherel measure by Okounkov and Pandharipande in~\cite{OPaGro}. This viewpoint naturally led to the conjecture that the Gromov--Witten invariants of $\bp^1$ satisfy topological recursion applied to the complex curve defined by a dual Landau--Ginzburg model \cite{NScGro}, proven in \cite{DBOSS} using a completely different point of view. Topological recursion, defined in Section~\ref{TR}, produces a collection of meromorphic multidifferentials $\omega_{g,n}$ on a complex curve via a recursive relation between the expansion of $\omega_{g,n}$ at its poles and the expansion of $\omega_{g',n'}$ for $2g'-2+n'<2g-2+n$ at their poles. This relation at the poles takes the form of Virasoro constraints which we will call {\em local} Virasoro constraints. The Gromov--Witten invariants of $\bp^1$ already satisfy Virasoro constraints, conjectured in~\cite{EHXTop} and proven in~\cite{GivGro,OPaVir}, which we will call {\em global} Virasoro constraints. Until now the direct relation between these local and global Virasoro constraints has been missing. This paper fills this gap, showing that the global Virasoro constraints satisfied by Gromov--Witten invariants of $\bp^1$ are a~consequence of the local Virasoro constraints that constitute topological recursion.

The genus $g$ connected Gromov--Witten invariants of $\bp^1$ with $n$ insertions are defined as
\begin{gather*}
\bigg\langle \prod_{i=1}^n \tau_{b_i}^{\alpha_i}\bigg\rangle_g^d=\int_{[\overline{\modm}_{g,n}(\bp^1,d)]^{{\rm vir}}}\prod_{i=1}^n \psi_i^{b_i} {\rm ev}_i^\ast(\gamma_{\alpha_i}),
\end{gather*}
where $d$ is determined by $\sum\limits_{i=1}^n b_i=2g-2+2d$ and hence sometimes omitted from the notation. The cohomology classes $\gamma_i$ are chosen to be either the unit $\gamma_0=1 \in H^0\big(\bp^1\big)$ (non-stationary insertion) or the dual of the class of a point $\gamma_1=\omega \in H^2\big(\bp^1\big)$ (stationary insertion). The $0$-point invariants vanish except for $1=\int_{[\overline{\modm}_{0,0}(\bp^1,1)]^{{\rm vir}}} 1$ so we consider only $n>0$. The generating series of stationary invariants
\begin{gather}\label{genstat}
\Omega_{g,n}(x_1,\ldots,x_n)=\sum_{b_1,\ldots,b_n \geq 0}\bigg\langle \prod_{i=1}^n \tau_{b_i}^1 \bigg\rangle_g \prod_{i=1}^n \frac{(b_i+1)!}{x_i^{b_i + 2}}
\end{gather}
is analytic in a neighbourhood of $x_i=\infty$ and for $(g,n)\neq (0,1)$ it analytically continues to a~meromorphic function on $\mathcal{S}^n$, for $\mathcal{S}\cong\bp^1$, via the substitution $x_i=x(z_i)=z_i+1/z_i$.

In order to prove the Virasoro constraints satisfied by the Gromov--Witten invariants of $\bp^1$, we represent $\big\langle \prod\limits_{i=1}^n \tau_{b_i}^{\alpha_i}\big\rangle_g^d$ as periods, or contour integrals, of the rational functions defined in~\eqref{genstat}. The computation of these invariants as periods is expected from mirror symmetry. The stationary invariants were already known to be given by periods of rational functions, and what is new here is representing the non-stationary invariants as contour integrals of the same rational functions. By construction, $\Omega_{g,n}(x_1,\ldots,x_n)$ stores the stationary invariants via
\begin{gather*}
\bigg\langle \prod_{i=1}^n \tau_{b_i}^{1}\bigg\rangle_g= \bigotimes_{i = 1}^n \mathcal{I}_{b_i}^1 \big[\Omega_{g,n}\dd x_1 \cdots \dd x_n\big],
\end{gather*}
where the linear functional $\ci_b^1$ is a contour integral defined on a meromorphic $1$-form $f$ by
\begin{gather*}
\ci_{b}^{1}[f] := -\Res_{\infty} \frac{x^{b + 1}}{(b + 1)!}\,f.
\end{gather*}
We show how $\Omega_{g,n}(x_1,\ldots,x_n)$ also stores the non-stationary invariants via contour integrals. The contours involved are now non-compact, so we need to first develop the main technical tool introduced in this paper, which is a collection of regularised contour integrals. Theorem~\ref{insertions} exhibits the non-stationary invariants as regularised contour integrals of the analytic continuation of $\Omega_{g,n}$. It is rather interesting that the non-stationary invariants use the global structure of the analytic continuation.

To state the result, we distinguish $\mathcal{S} \cong \bp^1$ with coordinate $z$, and $\check{\mathcal{S}} \cong \bp^1$ with coordinate $x$. If $f$ is a $1$-form on $\mathcal{S}$ without poles at $\infty$, we introduce for $b \geq 0$
\begin{gather*}
\mathcal{R}_{b }[f] = f - \sum_{a = 0}^{b-1} \frac{(a + 1)!\,\dd x}{x^{a + 2}}\,\mathcal{I}_{a}^{1}[f].
\end{gather*}
We define an inverse function to $x$ by
\begin{align*}
Z\colon \ \check{\mathcal{S}} \setminus [-2,2] & \longrightarrow \mathcal{S}, \\ x & \longmapsto \dfrac{x - \sqrt{x^2 - 4}}{2},
\end{align*}
with the standard determination of the square root having a discontinuity on $\mathbb{R}_{-}$. We then define $\check{\mathcal{R}}_{b}[f](x) = \mathcal{R}_{b}[f](Z(x))$ which is a $1$-form on $\check{\mathcal{S}} \setminus [-2,2]$. Notice that $\check{\mathcal{R}}_{b}[f]$ behaves as $O\big(x^{-(b + 1)}\dd x\big)$ when $x \rightarrow 0$ and behaves as $O\big(x^{-b}\dd x\big)$ when $x \rightarrow \infty$. This remark shows that the following definition is well-posed
\begin{gather} \label{I^0}
\mathcal{I}_{b}^{0}[f] = \lim_{\epsilon \rightarrow 0^+} \left(2\mathcal{I}_{b-1}^{1}[f]\,\ln \epsilon - \int_{0}^{{\rm i}\infty} \frac{x^{b }}{b!}\,\check{\mathcal{R}}_{b}[f](x + {\rm i}\epsilon) - \int_{0}^{-{\rm i}\infty} \frac{x^{b }}{b!}\,\check{\mathcal{R}}_{b}[f](x - {\rm i}\epsilon)\right).
\end{gather}
When $b=0$ we set $\ci_{-1}^1=0$. In Section~\ref{propint} it is proven that $\ci^0_{b}[f] = -\int_{\gamma} \frac{x^b}{b!}\,f$ for a certain class of $1$-forms $f$ on $\mathcal{S}$, where $\gamma$ is a contour from $z = 0$ to $z = \infty$ (the two points above $x=\infty$). Rather generally, $\ci^0_b$ satisfies a formula of integration by parts. Hence we consider the linear functional $\ci^0_b$ to be a regularised integral. We actually extend the definition of $\ci^0_{b}$ to $1$-forms with poles or other singularities at $\infty$ such that the integration by parts is still satisfied.

For $2g-2+n>0$, define $\omega_{g,n}(z_1,\ldots,z_n)$ to be the analytic continuation of
\begin{gather*}
\Omega_{g,n}(x_1,\ldots,x_n)\dd x_1\cdots \dd x_n
\end{gather*}
to $\mathcal{S}^n$ where $x_i=z_i+1/z_i$. Equivalently $\omega_{g,n}(z_1,\ldots,z_n)$ is rational with expansion around \smash{$z_i=\infty$} given by $\Omega_{g,n}(x_1,\ldots,x_n)\dd x_1 \cdots \dd x_n$.
The fact that $\Omega_{g,n}(x_1,\ldots,x_n)$ analytically continues to a rational function is a consequence of topological recursion proven in~\cite{DBOSS}~-- see Remark~\ref{ratcont}~-- or topological recursion relations obtained by pulling back relations in $H^*\big(\overline{\modm}_{g,n}\big)$ to relations in $H^*\big(\overline{\modm}_{g,n}\big(\bp^1,d\big)\big)$ \cite{NorSta} which are satisfied quite generally by Gromov--Witten invariants. We expect that it can also be derived from the semi-infinite wedge formalism of Okounkov and Pandharipande \cite{OPaGro}.

\begin{theorem} \label{insertions}For $2g - 2 + n > 0$, $b_1,\ldots,b_n \geq 0$ and $\alpha_1,\ldots,\alpha_n \in \{0,1\}$
\begin{gather} \label{insert}
\bigg\langle \prod_{i = 1}^{n} \tau_{b_i}^{\alpha_i} \bigg\rangle_g = \bigg(\bigotimes_{i = 1}^n \mathcal{I}_{b_i}^{\alpha_i}\bigg)\omega_{g,n}.
\end{gather}
In the case $(g,n)=(0,2)$, replace $\omega_{0,2}$ in \eqref{insert} by
\begin{gather*}
\omega^{\rm odd}_{0,2}(z_1,z_2)=\frac12\frac{\dd z_1\,\dd z_2}{(z_1-z_2)^2}+\frac{1}{2}\frac{\dd z _1 \dd z_2}{(1-z_1z_2)^2},
\end{gather*}
and for $(g,n)=(0,1)$ we have
\begin{gather*}
\big\langle\tau_{b}^{\alpha} \big\rangle_0 = \mathcal{I}_{b+1}^{\alpha}\left[\frac{\dd z}{z}\right].
\end{gather*}
\end{theorem}

\begin{remark}Integrals of differentials over compact and non-compact contours appearing in work of Dubrovin~\cite{DubGeo} were used in~\cite{DNOPSPri} to produce correlators of cohomological field theories such as Gromov--Witten invariants. That paper considered only {\em primary} invariants, corresponding to~$\ci^\alpha_0$, where the contour integrals do not require regularisation. Our technical contribution is a rigorous definition of the integrals over non-compact contours.
\end{remark}

Define the partition function which stores Gromov--Witten invariants by
\begin{gather*}
Z_{\bp^1}\big(\hbar,\big\{t^{\alpha}_k\big\}\big)=\exp\Bigg(\sum_{\substack{g,n,b_1,\ldots,b_n \geq 0 \\ \alpha_1,\ldots,\alpha_n \in \{0,1\}}} \frac{\hbar^{g-1}}{n!}\bigg\langle \prod_{i=1}^n\tau_{b_i}^{\alpha_i} \bigg\rangle_g \prod_{i = 1}^n t^{\alpha_i}_{b_i}\Bigg).
\end{gather*}
Define the Virasoro operators $L_n$ for $n\geq -1$ by
\begin{gather}
 L_n = - (n+1)! \frac{\partial}{\partial t^0_{n+1}} + \sum_{j\geq 1} \frac{(n+j)!}{(j-1)!}
t^0_j \frac{\partial}{\partial t^0_{n+j}} + 2\sum_{j\geq 1} \frac{(n+j)!}{(j-1)!} (H_{n+j} -H_{j-1})
t^0_{j} \frac{\partial}{\partial t^1_{n+j-1}} \nonumber\\
\hphantom{L_n =}{} + \sum_{j\geq 0} \frac{(n+j+1)!}{j!}t^1_j \frac{\partial}{\partial t^1_{n+j}} - 2(n+1)!H_{n+1}
\frac{\partial}{\partial t^1_n} \nonumber\\
\hphantom{L_n =}{} + \hbar\sum_{j = 0}^{n-2}
(j+1)! (n-j-1)! \frac{\partial}{\partial t^1_j}
\frac{\partial}{\partial t^1_{n-j-2}} + \hbar^{-1}\delta_{n,0} (t_0^0)^2+ \hbar^{-1}\delta_{n,-1} t_0^0t_0^1,\label{viror}
\end{gather}
where $H_k = \sum\limits_{j =1}^{k} \frac{1}{j}$. We give a new proof of the following theorem of~\cite{OPaVir}.
\begin{theorem}\label{th:vir}The partition function $Z_{\bp^1}\big(\big\{t^{\alpha}_k\big\}\big)$ satisfies the Virasoro constraints
\begin{gather*}
\forall\, n \geq -1,\qquad L_n\cdot Z_{\bp^1}\big(\big\{t^{\alpha}_k\big\}\big)=0.
\end{gather*}
\end{theorem}
We derive the global constraints of Theorem~\ref{th:vir} from the local constraints of topological recursion, by moving the contours from poles $z = \pm 1$ to $z = \infty$. For spectral curves that are smooth algebraic plane curves, there would be no difficulty in doing so, cf., e.g.,~\cite{OrantinVir}. The difficulty which we overcome is the handling of the log singularities in the Landau--Ginzburg model dual to~$\mathbb{P}^1$, via an asymptotic expansion result for the Hilbert transform found in~\cite{WongMcClure}. The definition of the regularised integral~\eqref{I^0} appears as a byproduct of our analysis. We expect that our method generalises to more complicated spectral curves with logarithmic singularities, in particular to certain types of Hurwitz numbers and Gromov--Witten invariants for some other targets.

The Virasoro constraints for Gromov--Witten invariants have been proven in some generality by Givental and Teleman \cite{GivGro,Tele}. The Virasoro operators can be obtained from conjugation of local Virasoro operators by operators that reconstruct the partition function of the Gromov--Witten invariants from the partition function of Gromov--Witten invariants of a point. The work of \cite{DBOSS} showed that topological recursion is equivalent to this reconstruction of Givental and Teleman. Hence one would expect that the global Virasoro constraints of Theorem~\ref{th:vir} can be derived directly from the local Virasoro constraints. Our result essentially exhibits this conjugation via moving contours.

The paper is organised as follows. In Section~\ref{sec:GW} we recall the definition of descendant and ancestor Gromov--Witten invariants which are needed in the proof of Theorem~\ref{insertions}. In Section~\ref{TR} we review the topological recursion and its application to Gromov--Witten invariants of $\bp^1$ proven in \cite{DBOSS}, and derive a preliminary form of the global constraints from the local ones. Section~\ref{regint} develops the regularised integral and its properties which is the main technical tool of this paper. Section~\ref{proofthm1} contains the proof of Theorem~\ref{insertions}. In Section~\ref{S7}, we use the properties of the regularised integral and the aforementioned preliminary global constraints to produce a new proof of Theorem~\ref{th:vir}.

\section{Gromov--Witten invariants} \label{sec:GW}

\subsection{The moduli space of stable maps}
Let $X$ be a projective algebraic variety and consider $(C,x_1,\ldots,x_n)$ a connected smooth curve of genus $g$ with $n$ distinct marked points. For $d \in H_2(X,\bz)$ the moduli space of maps $\modm_{g,n}(X,d)$ consists of morphisms
\begin{gather*}
f\colon \ (C,x_1,\ldots,x_n)\rightarrow X
\end{gather*}
satisfying $f_\ast [C]= d$ quotiented by isomorphisms of the domain $C$ that fix each $x_i$. The moduli space has a compactification $\overline{\modm}_{g,n}(X,d)$ given by the moduli space of stable maps: the domain~$C$ is a connected nodal curve; the distinct points $\{x_1,\ldots,x_n\}$ avoid the nodes; any genus zero irreducible component of $C$ with fewer than three distinguished points (nodal or marked) must be collapsed to a point; any genus one irreducible component of~$C$ with no marked point must be collapsed to a point. The moduli space of stable maps has irreducible components of different dimensions but it has a \textit{virtual fundamental class}, $\big[\overline{\modm}_{g,n}(X,d)\big]^{{\rm vir}}$, the existence and construction of which is highly nontrivial~\cite{BehrendFantechi}, of dimension
\begin{gather}
\label{dimeq} \dim\big[\overline{\modm}_{g,n}(X,d)\big]^{\text{vir}}=\langle c_1(X),d\rangle +(\dim X-3)(1-g)+n.
\end{gather}

\subsubsection[Cohomology on $\overline{\modm}_{g,n}(X,d)$]{Cohomology on $\boldsymbol{\overline{\modm}_{g,n}(X,d)}$}
Let $\mathcal{L}_i$ be the line bundle over $\overline{\modm}_{g,n}(X,d)$ with fibre at each point the cotangent bundle over the $i$th marked point of the domain curve $C$. Define $\psi_i=c_1(\mathcal{L}_i)\in H^2\big(\overline{\modm}_{g,n}(X,d),\bq\big)$ to be the first Chern class of~$\mathcal{L}_i$. For $i \in \{1,\ldots,n\}$ there exist evaluation maps
\begin{gather*}
{\rm ev}_i\colon \ \overline{\modm}_{g,n}(X,d)\longrightarrow X, \qquad {\rm ev}_i(f)=f(x_i),
\end{gather*}
so that classes $\gamma\in H^*(X,\bz)$ pull back to classes in $H^*\big(\overline{\modm}_{g,n}(X,d),\bq\big)$
\begin{gather*}
{\rm ev}_i^\ast\colon \ H^*(X,\bz)\longrightarrow H^*(\overline{\modm}_{g,n}(X,d),\bq).
\end{gather*}
The forgetful map $\pi\colon \overline{\modm}_{g,n}(X,d)\to\overline{\modm}_{g,n}$ sends the map to its domain curve with possible contractions of unstable components.

The Gromov--Witten invariants are defined by integrating cohomology classes, often called descendant classes, of the form
\begin{gather*}
\tau_{b_i}(\gamma)=\psi_i^{b_i}{\rm ev}^\ast_i(\gamma)
\end{gather*}
against the virtual fundamental class. The descendant Gromov--Witten invariants are defined by
\begin{gather*}
\bigg\langle \prod_{i=1}^n\tau_{b_i}(\gamma_{\alpha_i}) \bigg\rangle_g^{d}:=\int_{[\overline{\modm}_{g,n}(X,d)]^{{\rm vir}}} \prod_{i=1}^n\psi_i^{b_i} {\rm ev}_i^\ast(\gamma_{\alpha_i}).
\end{gather*}
%For any $I \subseteq \{1,\ldots,n\}$, we use the convention $\tau_{\bm{b}}^{\aalpha_I}=\prod\limits_{i \in I} \tau_{b_i}^{\alpha_i}$.
When $X=\{pt\}$, $\overline{\modm}_{g,n}(X,d)=\overline{\modm}_{g,n}$ is the moduli space of genus $g$ stable curves with $n$ labeled points, equipped with line bundles $\mathcal{L}_i$ with fibre at each point the cotangent bundle over the $i$th marked point of the domain curve $C$ and $\psi_i=c_1(\mathcal{L}_i)\in H^2\big(\overline{\modm}_{g,n},\bq\big)$. For $2g-2+n>0$, let $\overline{\psi}_i=\pi^*\psi_i\in H^2\big(\overline{\modm}_{g,n}(X,d),\bq\big)$. The ancestor Gromov--Witten invariants use the classes $\overline{\psi}_i$ in place of~$\psi_i$:
\begin{gather*}
\bigg\langle \prod_{i=1}^n\overline{\tau}_{b_i}(\gamma_{\alpha_i}) \bigg\rangle _{g}^{d}:=\int_{[\overline{\modm}_{g,n}(X,d)]^{{\rm vir}}} \prod_{i=1}^n\overline{\psi}_i^{b_i} {\rm ev}_i^\ast(\gamma_{\alpha_i}).
\end{gather*}
They are defined only in the stable case, i.e., when $2g-2+n>0$.

\subsection[Specialising to $\bp^1$]{Specialising to $\boldsymbol{\bp^1}$}
We now only consider the target $X = \bp^1$. Let $\gamma_0 \in H^0\big(\bp^1,\mathbb{Z}\big)$ be the unit
and $\gamma_1 \in H^2\big(\bp^1,\mathbb{Z}\big)$ be the Poincar\'e dual class of a point. For brevity we denote $\tau_{b}^{\alpha} := \tau_{b}(\gamma_{\alpha})$. The degree $d \in \mathbb{N}$ of the Gromov--Witten invariants $\big\langle \prod\limits_{i = 1}^n \tau_{b_i}^{\alpha_i} \big\rangle_g$ is determined by $\sum\limits_{i=1}^n (b_i+\alpha_i)=2g-2+2d+n$ coming from~\eqref{dimeq}. Insertions of $\gamma_1$ are called stationary Gromov--Witten invariants of~$\bp^1$ since the images of the marked points are fixed, and insertions of $\gamma_0$ are called non-stationary. The ancestor invariants of $\bp^1$ use the analogous notation $\big\langle \prod\limits_{i=1}^n\overline{\tau}_{b_i}^{\alpha_i}\big\rangle_g$. We introduce the descendant partition function in the variables $t^{\alpha}_k$ for $\alpha\in\{0,1\}$ and $k\in\bn$ by
\begin{gather} \label{part}
Z_{\bp^1}\big(\hbar,\big\{t^{\alpha}_k\big\}\big)=\exp\Bigg(\sum_{\substack{g,n,b_1,\ldots,b_n \geq 0 \\ \alpha_1,\ldots,\alpha_n \in \{0,1\} \\ 2g - 2 + n > 0}} \frac{\hbar^{g-1}}{n!}\bigg\langle \prod_{i=1}^n\tau_{b_i}^{\alpha_i} \bigg\rangle_g \prod_{i = 1}^n t^{\alpha_i}_{b_i}\Bigg),
\end{gather}
and the ancestor partition function using the variables $\overline{t}^{\alpha}_k$:
\begin{gather} \label{partanc}
\overline{Z}_{\bp^1}\big(\hbar,\big\{\overline{t}^{\alpha}_k\big\}\big)=\exp\Bigg(\sum_{\substack{g,n,b_1,\ldots,b_n \geq 0 \\ \alpha_1,\ldots,\alpha_n \in \{0,1\}}} \frac{\hbar^{g-1}}{n!}\bigg\langle \prod_{i=1}^n\overline{\tau}_{b_i}^{\alpha_i} \bigg\rangle_g\prod_{i = 1}^n \overline{t}^{\alpha_i}_{b_i}\Bigg).
\end{gather}
The descendant invariants uniquely determine the ancestor invariants. They are related by an endomorphism valued series $S(u)=\sum\limits_{k \geq 0} S_k\,u^k$ known as the $S$-matrix, by the linear change
\begin{gather*}
\overline{t}_{m}^{\alpha} = \sum_{\substack{k \geq m \\ \beta \in \{0,1\}}} (S_{k - m})_{\beta}^{\alpha}t_{k}^{\beta}.
\end{gather*}
Equivalently
\begin{gather} \label{Srel}
\tau_k^{\beta} = \sum_{\substack{m \leq k \\ \alpha \in \{0,1\}}} (S_{k- m})_{\beta}^{\alpha} \overline{\tau}_{m}^{\alpha},
\end{gather}
when evaluated between $\langle\cdot\rangle$, so for example the 1-point genus $g$ invariants satisfy
\begin{gather*}
\big\langle \tau_{k}^{\beta} \big\rangle_{g} = \sum_{\substack{m \leq k \\ \alpha \in \{0,1\}}} (S_{k - m})_{\beta}^{\alpha} \big\langle \overline{\tau}_{m}^{\alpha}\big\rangle_{g}.
\end{gather*}
It is proven in \cite{KMaRel} that
\begin{gather*}
Z_{X}^{{\rm st}}\big(\hbar,\big\{t_{k}^{\beta}\big\}\big) =
\overline{Z}_{X}\big(\hbar,\big\{\overline{t}_{m}^{\alpha}\big\}\big)\big|_{\overline{t}_{m}^{\alpha}
= \sum\limits_{k,\beta} (S_{k - m})_{\beta}^{\alpha} t_{k}^{\beta}},
\end{gather*}
where $Z^{{\rm st}}_X\big(\h,\big\{t^{\beta}_k\big\}\big)$ is the stable part of the descendant invariants, i.e., it excludes the terms $(g,n) = (0,1)$ and $(0,2)$ in~\eqref{part}.

\section{Topological recursion} \label{TR}

\subsection{Definition}

Topological recursion \cite{EORev} takes as input a spectral curve $\mathcal{C} = (\mathcal{S},x, y,B)$ consisting of a Riemann surface $\mathcal{S}$, two meromorphic functions $x$ and $y$ on $\mathcal{S}$ and a symmetric bidifferential $B$ on $\mathcal{S}^2$. We assume that each zero of $\dd x$ is simple and does not coincide with a zero of $\dd y$. The output of topological recursion is a collection of symmetric multidifferentials $\omega_{g,n}\in H^0\big(K_{\mathcal{S}}(*D)^{\boxtimes n},\mathcal{S}^n\big)^{\mathfrak{S}_{n}}$ for $g \geq 0$ and $n \geq 1$ such that $2g - 2 + n > 0$, which we call \emph{correlators}. Here $D$ is the divisor of zeroes of $\dd x=0$. In other words the multidifferentials are holomorphic outside of $\dd x=0$ and can have poles of arbitrary order when each variable approaches $D$.

The correlators are defined as follows. We first define the exceptional cases
\begin{gather*}
\omega_{0,1}(p_1) = y(p_1) \, \dd x(p_1) \qquad \text{and} \qquad \omega_{0,2}(p_1, p_2) = B(p_1,p_2).
\end{gather*}
The correlators $\omega_{g,n}$ for $2g-2+n>0$ are defined recursively via the following equation
\begin{gather*}
 \omega_{g,n}(p_1, \pp_I) = \sum_{\dd x(\alpha) = 0} \mathop{\text{Res}}_{p=\alpha} K(p_1, p) \Bigg[ \omega_{g-1,n+1}(p, \sigma_{\alpha}(p), \pp_I) \\
 \hphantom{\omega_{g,n}(p_1, \pp_I) =}{} + \sum_{\substack{h+h'=g \\ J \sqcup J' = I}}^{\circ} \omega_{h,1 + |J|}(p, \pp_J) \, \omega_{h',1+|J'|}( \sigma_{\alpha}(p), \pp_{J'}) \Bigg].
\end{gather*}
Here, we use the notation $I = \{2, 3, \ldots, n\}$ and $\pp_J = \{p_{j_1}, p_{j_2}, \ldots, p_{j_k}\}$ for $J = \{j_1, j_2, \ldots, j_k\}\allowbreak \subseteq I$. The holomorphic function $p \mapsto \sigma_{\alpha}(p)$ is the non-trivial involution defined locally at the ramification point $\alpha$ and satisfying $x( \sigma_{\alpha}(p)) = x(p)$. The symbol $\circ$ over the inner summation means that we exclude any term that involves $\omega_{0,1}$. Finally, the recursion kernel is given by
\begin{gather*}
K(p_1,p) = \frac{1}{2}\frac{\int_{\sigma_{\alpha}(p)}^{p} \omega_{0,2}(p_1,\cdot)}{[y(p) - y( \sigma_{\alpha}(p))]\dd x(p)}.
\end{gather*}
The recursion only depends on the local behaviour of $y$ near the zeros of $\dd x$ up to functions that are even with respect to the involution. Hence it only depends on $\dd y$. Below we write a spectral curve~\eqref{P1spec} in terms of~$\dd y$.

For $2g-2+n>0$, the multidifferentials $\omega_{g,n}(p_1,\ldots,p_n)$ are meromorphic on $\mathcal{S}^n$ with poles at $p_i \in D$. They can be expressed as polynomials in a basis of differentials with poles only at~$D$ and divergent part odd under each local involution~$\sigma_{\alpha}$. We denote $\xi^\alpha_{k}$ such a basis indexed by $k \geq 0$ and zeroes $\alpha$ of~$\dd x$. Once a choice of basis is made, we define the partition function of the spectral curve $\mathcal{C}=(\mathcal{S},x,y,B)$ by
\begin{gather*}
\label{partpt}Z^{\mathcal{C}}\big(\hbar,\big\{t^{\alpha}_k\big\}\big)=\exp\Bigg(\sum_{\substack{g \geq 0,\, n \geq 1 \\ 2g - 2 + n > 0}} \frac{\hbar^{g-1}}{n!}\omega_{g,n}|_{\xi^\alpha_k=t^\alpha_k}\Bigg).
\end{gather*}

\subsection[Relation to Gromov--Witten theory of $\bp^1$]{Relation to Gromov--Witten theory of $\boldsymbol{\bp^1}$}

Dunin-Barkowski, Orantin, Shadrin and Spitz \cite{DBOSS} proved that for a particular choice of basis~$\big\{\xi_k^{\alpha}\big\}$, the partition function $Z^{\mathcal{C}}\big(\hbar,\big\{t^{\alpha}_k\big\}\big)$ coincides with the partition function of a cohomological field theory. In particular, they showed how to realise in this way the cohomological field theory encoding Gromov--Witten invariants of~$\bp^1$, which corresponds to the spectral curve
\begin{gather} \label{P1spec}
\mathcal{C}_{\bp^1} =\left(\bp^1,\,x=z+\frac{1}{z},\,\dd y=\frac{\dd z}{z},\,B=\frac{\dd z_1\,\dd z_2}{(z_1-z_2)^2}\right).
\end{gather}
For $2g-2+n>0$, the associated correlators $\omega_{g,n}$ have poles at $D = \{-1,1\}$ and the global involution $z\mapsto 1/z$ realises the local involutions $\sigma_{\pm 1}$. The aforementioned basis of $1$-forms is defined by induction for $k \geq 0$ and $\alpha \in \{0,1\}$
\begin{gather} \label{auxdif}
\xi^\alpha_{k}(z)=-\dd\left(\frac{\xi^\alpha_{k-1}(z)}{\dd x(z)}\right),
\end{gather}
from the initial data $\xi^\alpha_{-1}(z)= z^{-\alpha}\,\dd z$ which are not part of the basis.

\begin{remark} \label{oddre} The $\xi^\alpha_{k}(z)$ for $k \geq 0$ are odd under the involution because they are unchanged if we replace $\xi^\alpha_{-1}(z)$ with its odd part
\begin{gather}
\label{xiodd}\xi^{\alpha,{\rm odd}}_{-1}(z)=\left(\frac{x}{2}\right)^{1-\alpha}\frac{\dd z}{z}.
\end{gather}
\end{remark}

\begin{theorem}[\cite{NScGro} for $g \in \{0,1\}$, \cite{DBOSS} in general]\label{th1p}
For $2g - 2 + n > 0$, $\Omega_{g,n}(x(z_1),\ldots,x(z_n))\dd x(z_1)\otimes \cdots \otimes\dd x(z_n)$ initially defined as a formal series near $z_i = \infty$ analytically continues to a symmetric multidifferential on $\mathcal{S}^n$, which coincides with the correlators of the topological recursion for the spectral curve \eqref{P1spec}. In particular,
\begin{gather*}
\bigg\langle \prod_{i = 1}^n \tau_{b_i}^1 \bigg\rangle_g = \mathcal{I}_{b_1}^{1} \otimes \cdots \otimes\mathcal{I}_{b_n}^{1}[\omega_{g,n}].
\end{gather*}
\end{theorem}

For $2g-2+n>0$, each correlator $\omega_{g,n}$ of $\mathcal{C}_{\bp^1}$ is a polynomial in $\xi^\alpha_{k}(z)$.
Using the basis in~\eqref{auxdif}, the topological recursion partition function of the spectral curve $\mathcal{C}_{\bp^1}$ coincides with the ancestor partition function~\eqref{partanc} of the Gromov--Witten invariants of $\bp^1$.
\begin{theorem}[\cite{DBOSS}] \label{th:dboss}
Let $\omega_{g,n}$ be the correlators of the topological recursion applied to the spectral curve~$\mathcal{C}$ defined by~\eqref{P1spec}. Then
\begin{gather*}
\frac{\overline{Z}_{\bp^1}\big(\hbar,\big\{\overline{t}^{\alpha}_k\big\}\big)}{\overline{Z}_{\bp^1}(\hbar,0)}=Z^{\mathcal{C}_{\bp^1}}\big(\hbar,\big\{\overline{t}^{\alpha}_k\big\}\big)=\exp\Bigg(\sum_{\substack{g \geq 0,\, n \geq 1 \\ 2g-2+n>0}}\frac{\hbar^{g-1}}{n!}\omega_{g,n}|_{\xi^\alpha_k=\overline{t}^\alpha_k}\Bigg).
\end{gather*}
\end{theorem}

Theorem~\ref{th:dboss} states that the coefficients of the differentials $\xi^{0}_m$ and~$\xi^{1}_m$ correspond to insertions of stationary ancestor invariants $\overline{\tau}^0_m$, respectively~$\overline{\tau}^1_m$. To retrieve the descendant Gromov--Witten invariants from the correlators $\omega_{g,n}$ one must understand elements of the dual of the space of meromorphic differentials on the spectral curve which can be naturally realised as integration over contours on the spectral curve. The next section is devoted to technical aspects of integration over non-compact contours which need regularisation.
\begin{remark} \label{ratcont}
The proof of Theorem~\ref{th1p} in~\cite{DBOSS} uses Theorem~\ref{th:dboss} together with the linear functio\-nals~$\ci^1_k$ which are shown to encode the $S$-matrix coefficients required to produce stationary descendant invariants. An immediate consequence is that the Taylor expansion of $\omega_{g,n}(z_1,\ldots,z_n)$ around $z_i=\infty$ with respect to the local coordinate $1/x(z)$ gives $\Omega_{g,n}(x_1,\ldots,x_n)\dd x_1 \cdots \dd x_n$ when $2g-2+n>0$. In particular this gives a proof that there is an analytic continuation of $\Omega_{g,n}(x_1,\ldots,x_n)\dd x_1\cdots \dd x_n$ to a rational curve.
\end{remark}

\subsection{From local to global constraints on multidifferentials}

Let $\omega_{g,n}$ be the multidifferentials of the topological recursion for the spectral curve \eqref{P1spec}. We choose the determination of the logarithm to have a branchcut on ${\rm i}\mathbb{R}_{-}$, and such that $\ln(1) = 0$. With this choice, $y(z)$ is holomorphic in the neighborhood of $z = \pm 1$.

The topological recursion is such that for any $g \geq 0$ and $n \geq 1$, $\omega_{g,n}$ satisfy the linear and quadratic loop equations \cite{BEO}. The linear loop equations is a symmetry property with respect to the involution $z \mapsto 1/z$
\begin{gather}
\label{linearloopeq} \omega_{g,n}(z,z_2,\ldots,z_n) + \omega_{g,n}(1/z,z_2,\ldots,z_n) = \delta_{g,0}\delta_{n,2}\,\frac{\dd x(z_1)\dd x(z_2)}{(x(z_1) - x(z_2))^2}.
\end{gather}
The quadratic loop equations state that, for $I = \{2,\ldots,n\}$ and denoting $\dd x(\zz_{I}) = \prod\limits_{i = 2}^n \dd x(z_i)$,
\begin{gather}
\label{qgns} q_{g,n}(z;\zz_I) = \frac{1}{\dd x(z)^2}\Bigg(\omega_{g - 1,n + 1}(z,1/z,\zz_I) + \sum_{\substack{h + h' = g \\ J \sqcup J' = I}} \omega_{h,1 + |J|}(z,\zz_J)\omega_{h',1 + |J'|}(1/z,\zz_{J'})\Bigg)
\end{gather}
is holomorphic in a neighborhood of $z = \pm 1$. These two sets of equations are (by definition) equivalent to the local Virasoro constraints mentioned in the introduction. We would like to derive from them global constraints, which concern the Laurent expansion of $q_{g,n}$ at $x(z) = \infty$. We use the following notation.
\begin{definition}If $f$ is a $1$-form, we define
\begin{gather}
\label{Ldgnfgfdg} \mathcal{A}[f](x_1) := \sum_{a = {\rm pole}\,\,{\rm of}\,\,f} \Res_{z = a} \frac{f(z)\,\ln z}{x_1 - x(z)},
\\
\label{Ldefdede} \mathcal{L}[f](z_1,z_i) := 2\,\omega_{0,2}^{{\rm odd}}(z_1,z_i)\,\frac{f(z_1)}{\dd x_1} - \dd_i\left(\frac{\dd x_1}{x_1 - x_i}\,\frac{f(z_i)}{\dd x_i}\right),
\end{gather}
where
\begin{gather*}
\omega_{0,2}^{{\rm odd}}(z_1,z_2) := \frac{1}{2}\big(\omega_{0,2}(z_1,z_2) - \omega_{0,2}(1/z_1,z_2)\big) = \omega_{0,2}^{{\rm odd}}(1/z_1,1/z_2).
\end{gather*}
\end{definition}

The following result gives a preliminary form of global constraints, which will be exploited in Section~\ref{S7}.
\begin{lemma}\label{Vircomp} Assume $2g - 2 + n \geq 2$. Let $I := \{2,\ldots,n\}$ and for any $i \in I$, set $I_i := I \setminus \{i\}$. We have
\begin{gather*}
 \mathcal{A}\big[\omega_{g,n}(\cdot,\zz_I)\big](x_1)\,\dd x_1
 = \sum_{i = 2}^n \mathcal{L}\big[\omega_{g,n - 1}(\cdot,\zz_{I_i})\big](z_1,z_i)+ \frac{\omega_{g - 1, n + 1}(z_1,z_1,\zz_I)}{\dd x_1} \\
\hphantom{\mathcal{A}\big[\omega_{g,n}(\cdot,\zz_I)\big](x_1)\,\dd x_1=}{} + \sum_{\substack{h + h' = g \\ J \sqcup J' = I}}^{\circ\circ} \frac{\omega_{h,1 + |J|}(z_1,\zz_J)\omega_{h',1 + |J'|}(z_1,\zz_{J'})}{\dd x_1}.
\end{gather*}
Here, ${}^{\circ\circ}$ means that the terms involving $\omega_{0,1}$ and $\omega_{0,2}$ are excluded from the sum. For $(g,n) = (0,3)$, we have
\begin{gather*}
\mathcal{A}[\omega_{0,3}(\cdot,z_2,z_3)](x_1)\,\dd x_1 = - \frac{2\omega_{0,2}^{{\rm odd}}(z_1,z_2)\omega_{0,2}^{{\rm odd}}(z_1,z_3)}{\dd x_1} \\
\hphantom{\mathcal{A}[\omega_{0,3}(\cdot,z_2,z_3)](x_1)\,\dd x_1 =}{} + \dd_2\left(\frac{\dd x_1}{x_1 - x_2}\,\frac{\omega_{0,2}^{{\rm odd}}(z_2,z_3)}{\dd x_2}\right) + \dd_3\left(\frac{\dd x_1}{x_1 - x_3}\,\frac{\omega_{0,2}^{{\rm odd}}(z_2,z_3)}{\dd x_3}\right).
\end{gather*}
For $(g,n) = (1,1)$, we have $\mathcal{A}[\omega_{1,1}](x_1)\,\dd x_1 = -\omega_{0,2}(z_1,1/z_1)$.
\end{lemma}

\begin{proof}Since $q_{g,n}$ is holomorphic in a neighborhood of $z = \pm 1$ we have
\begin{gather}
\label{qgndd} 0=\sum_{a = -1,1} \Res_{z = a} \frac{\dd x_1\,\dd x(z)}{x_1 - x(z)}\,q_{g,n}(z,z_2,\ldots,z_n),
\end{gather}
for $2g - 2 + n > 0$ and $x_1 \neq\pm 2$. To prove the lemma we are going to compute separately the contributions of the various terms in the right-hand side of \eqref{qgndd}.

 \textbf{Stable terms.} We observe that
\begin{gather*}
\omega_{g - 1,n + 1}(z,1/z,\zz_I) = -\omega_{g,n}(z,z,\zz_I) = \omega_{g,n}(1/z,1/z,\zz_I),
\end{gather*}
and $\omega_{g - 1,n + 1}(z,z,\zz_I) \in O\big(\dd x(z)^2/x(z)^{4}\big)$ when $z \rightarrow \infty$. Therefore, after division by $(x(z) - x_1)\dd x(z)$ this expression has no residues at $z = 0$ and $\infty$. We compute
\begin{align*}
A_{g,n}^{(g - 1,n + 1)} & := \sum_{a = -1,1} \Res_{z = a} \frac{\dd x_1}{x_1 - x(z)}\,\frac{\omega_{g - 1,n + 1}(z,1/z,\zz_I)}{\dd x(z)} \\
 & = \sum_{a = z_1,1/z_1} \Res_{z = a} \frac{\dd x_1}{x(z) - x_1}\,\frac{\omega_{g - 1,n + 1}(z,1/z,\zz_I)}{\dd x(z)} \\
 & = \frac{2 \omega_{g - 1,n + 1}(z_1,1/z_1,\zz_I)}{\dd x_1},
\end{align*}
noticing there are no contributions from $0$ and $\infty$ when moving the contour. In case $(g,n) \neq (1,1)$ this is also equal to
\begin{gather*}
A_{g,n}^{(g-1,n + 1)} = -\frac{2\omega_{g - 1,n + 1}(z_1,z_1,\zz_I)}{\dd x_1}.
\end{gather*}

Likewise if $h + h' = g$ and $J \sqcup J' = I$ such that $2h- 2 + (1 + |J|) > 0$ and $2h' - 2 + (1 + |J'|) > 0$, we compute
\begin{align*}
A_{g,n}^{(h,J),(h',J')} & := \sum_{a=-1,1}\Res_{z = a} \frac{\dd x_1}{x_1 - x(z)}\,\frac{\omega_{h,1 + |J|}(z,\zz_J)\omega_{h',1 + |J'|}(1/z,\zz_{J'})}{\dd x(z)} \\
& = \frac{2\omega_{h,1 + |J|}(z_1,\zz_J)\omega_{h',1 + |J'|}(1/z_1,\zz_{J'})}{\dd x_1} = -\frac{2\omega_{h,1 + |J|}(z_1,\zz_J)\omega_{h',1 + |J'|}(z_1,\zz_{J'})}{\dd x_1}.
 \end{align*}

\textbf{The $\boldsymbol{(0,2)\times (g,n - 1)}$ term.} Assume that $(g,n) \neq (0,3)$. Fix $i \in \{2,\ldots,n\}$ and let $I_i = I\setminus\{i\}$. We would like to compute, following the previous steps,
\begin{gather*}
A^{(g,n - 1)}_{g,n} := \sum_{a=-1,1}\Res_{z = a} \frac{\dd x_1}{x_1 - x(z)}\,\frac{\omega_{g,n - 1}(z,\zz_{I_i})\omega_{0,2}(1/z,z_i) + \omega_{g,n - 1}(1/z,\zz_{I_i})\omega_{0,2}(z,z_i)}{\dd x(z)},
\end{gather*}
but there are two notable differences. Firstly, there is a shift in the antisymmetry relation for~$\omega_{0,2}$
\begin{gather*}
\omega_{0,2}(z,z_i) + \omega_{0,2}(1/z,z_i) = \frac{\dd x(z)\dd x(z_i)}{(x(z) - x(z_i))^2}.
\end{gather*}
Secondly, the presence of $\omega_{0,2}$ creates a pole at $z = z_i$ and $1/z_i$. We obtain
\begin{align*}
A^{(g, n -1)}_{g,n} & = \sum_{a=z_i,1/z_i,z_1,1/z_1}\!\!\!\!\!\!\!\!\!\Res_{z = a}\!\frac{\dd x_1}{x(z) - x_1}\frac{-2\omega_{g,n - 1}(z,\zz_{I_i})\omega_{0,2}(z,z_i) + \omega_{g,n - 1}(z,\zz_{I_i}) \frac{\dd x(z)\dd x(z_i)}{(x(z) - x(z_i))^2}}{\dd x(z)} \\
& = \sum_{a = z_i,z_1} \Res_{z = a} \frac{2 \dd x_1}{x(z) - x_1}\frac{-2\omega_{g,n - 1}(z,\zz_{I_i})\omega_{0,2}(z,z_i) + \omega_{g,n - 1}(z,\zz_{I_i}) \frac{\dd x(z)\dd x(z_i)}{(x(z) - x(z_i))^2}}{\dd x(z)},
\end{align*}
since the integrand is again invariant under $z \mapsto 1/z$. There is a double pole at $z = z_i$ and a~simple pole at $z = z_1$. We obtain
\begin{align*}
A^{(g,n - 1)}_{(g,n)} & = 2 \dd_i\left(\frac{\omega_{g,n - 1}(\zz_{I})\,\dd x_1}{\dd x_i (x_1 - x_i)}\right) + 2 \frac{\omega_{g,n - 1}(z_1,\zz_{I_i})}{\dd x_1}\left(-2\omega_{0,2}(z_1,z_i) + \frac{\dd x_1\dd x_i}{(x_1 - x_i)^2}\right) \\
& = 2 \dd_i\left(\frac{\omega_{g,n - 1}(\zz_{I})\,\dd x_1}{\dd x_i (x_1 - x_i)}\right) - 2\big(\omega_{0,2}(z_1,z_i) - \omega_{0,2}(z_1,1/z_i)\big) \frac{\omega_{g,n - 1}(z_1,\zz_{I_i})}{\dd x_1} \\
 & = 2 \dd_i\left(\frac{\omega_{g,n - 1}(\zz_{I})\,\dd x_1}{\dd x_i (x_1 - x_i)}\right) + 2 \omega_{0,2}(1/z_1,z_i) \omega_{g,n - 1}(z_1,\zz_{I_i}) \\
 & \quad {} + 2 \omega_{0,2}(z_1,z_i)\omega_{g,n - 1}(1/z_1,\zz_{I_i}).
\end{align*}

 \textbf{$\boldsymbol{(0,2) \times (0,2)}$ term for the $\boldsymbol{(0,3)}$ case.} We have to consider
\begin{equation*}
\begin{split}
A_{0,3}^{(0,2),(0,2)} & := \sum_{a=-1,1}\Res_{z = a} \frac{\dd x_1}{x_1 - x(z)}\,\frac{\omega_{0,2}(z,z_2)\omega_{0,2}(1/z,z_3) + \omega_{0,2}(1/z,z_2)\omega_{0,2}(z,z_3)}{\dd x(z)} \\
& = \sum_{a=-1,1}\Res_{z = a} \frac{\dd x_1}{x_1 - x(z)}\,\frac{\omega_{0,2}(z,z_2)\omega_{0,2}(1/z,z_3)}{\dd x(z)} + (z_2 \leftrightarrow z_3) \\
& = \sum_{a = z_1,1/z_1,z_2,1/z_3}\Res_{z = a}\frac{\dd x_1}{x(z) - x_1}\,\frac{\omega_{0,2}(z,z_2)\omega_{0,2}(1/z,z_3)}{\dd x(z)} + (z_2 \leftrightarrow z_3).
\end{split}
\end{equation*}
In the first term, we have a simple pole at $z = z_1,1/z_1$ and double poles at $z = z_2$ and $z = 1/z_3$. Using $\omega_{0,2}(1/z,z_3) = \omega_{0,2}(z,1/z_3)$, we get
\begin{gather*}
A_{0,3}^{(0,2),(0,2)} = \frac{\omega_{0,2}(z_1,z_2)\omega_{0,2}(1/z_1,z_3) + \omega_{0,2}(1/z_1,z_2)\omega_{0,2}(z_1,z_3)}{\dd x_1} \\
\hphantom{A_{0,3}^{(0,2),(0,2)} =}{} + \dd_{2}\left(\frac{\dd x_1}{\dd x_2}\,\frac{\omega_{0,2}(z_2,1/z_3)}{x_2 - x_1}\right) + \dd_{3}\left(\frac{\dd x_1}{\dd x_3} \frac{\omega_{0,2}(z_2,1/z_3)}{x_3 - x_1}\right) + (z_2 \leftrightarrow z_3).
\end{gather*}
This can be written in terms of the odd part of $\omega_{0,2}$
\begin{gather*}
A_{0,3}^{(0,2),(0,2)} = -\frac{2\omega_{0,2}^{{\rm odd}}(z_1,z_2)\omega_{0,2}^{{\rm odd}}(z_1,z_3)}{\dd x_1} + \dd_2\left(\frac{\dd x_1}{\dd x_2}\,\frac{\omega_{0,2}^{{\rm odd}}(z_2,z_3)}{x_1 - x_2}\right) \\
\hphantom{A_{0,3}^{(0,2),(0,2)} =}{} + \dd_3\left(\frac{\dd x_1}{\dd x_3} \frac{\omega_{0,2}^{{\rm odd}}(z_2,z_3)}{x_1 - x_3}\right)
 + \frac{\dd x_1\,\dd x_2\,\dd x_3}{2(x_1 - x_2)^2(x_1 - x_3)^2}\\
\hphantom{A_{0,3}^{(0,2),(0,2)} =}{} - \dd_2\left(\frac{\dd x_1\,\dd x_3}{2(x_1 - x_2)(x_2 - x_3)^2}\right) - \dd_3\left(\frac{\dd x_1\,\dd x_2}{2(x_1 - x_3)(x_2 - x_3)^2}\right)
 + (z_2 \leftrightarrow z_3) \\
\hphantom{A_{0,3}^{(0,2),(0,2)}}{} = -\frac{4\omega_{0,2}^{{\rm odd}}(z_1,z_2)\omega_{0,2}^{{\rm odd}}(z_1,z_3)}{\dd x_1} + 2\dd_2\left(\frac{\dd x_1}{\dd x_2} \frac{\omega_{0,2}^{{\rm odd}}(z_2,z_3)}{x_1 - x_2}\right) \\
\hphantom{A_{0,3}^{(0,2),(0,2)} =}{}+ 2\dd_3\left(\frac{\dd x_1}{\dd x_3} \frac{\omega_{0,2}^{{\rm odd}}(z_2,z_3)}{x_1 - x_3}\right).
\end{gather*}

 \textbf{The $\boldsymbol{(0,1)}$ terms.} The last term is
\begin{gather*}
A_{g,n}^{(0,1),(g,n)} := \sum_{a=-1,1}\Res_{z =a} \frac{\dd x_1}{x_1 - x(z)} \frac{\omega_{g,n}(z,\zz_I)\omega_{0,1}(1/z) + \omega_{g,n}(1/z,\zz_I)\omega_{0,1}(z)}{\dd x(z)}.
\end{gather*}
Using the involution $z \mapsto 1/z$ and recalling that $\omega_{0,1}(z) = \ln z\,\dd x(z)$, we rewrite it for $2g - 2 + n \allowbreak > 0$ as
\begin{gather}
\label{integA} A_{g,n}^{(0,1),(g,n)} = -2\sum_{a=-1,1}\Res_{z =a} \frac{\dd x_1}{x_1 - x(z)} \frac{\omega_{g,n}(z,\zz_I)\omega_{0,1}(z)}{\dd x(z)} = 2 \mathcal{A}\big[\omega_{g,n}(\cdot,\zz_I)\big](x_1).
\end{gather}

This exhausts the study of the terms contributing to \eqref{qgndd}. Summing them up concludes the proof of the lemma.
\end{proof}

\section[Properties of $\mathcal{A}$ and $\mathcal{L}$]{Properties of $\boldsymbol{\mathcal{A}}$ and $\boldsymbol{\mathcal{L}}$}

The contribution of unstable $\omega$s in the global constraints of Lemma~\ref{Vircomp} is more complicated than the others and need special care. This technical section establishes their properties, that will be used in Section~\ref{S7}.

\subsection[Laurent expansion of $\protect{\mathcal{A}[f]}$]{Laurent expansion of $\boldsymbol{\mathcal{A}[f]}$}

We are going to compute the Laurent series expansion of $\mathcal{A}[f](x_1)$, defined in \eqref{Ldgnfgfdg}, when \smash{$x_1 \rightarrow \infty$}, where $f$ is a meromorphic $1$-form on $\mathcal{S}$ with poles away from $z = 0,\infty$ and such that $f(z) + f(1/z) = 0$.

\begin{figure}\centering
\includegraphics[width=0.47\textwidth]{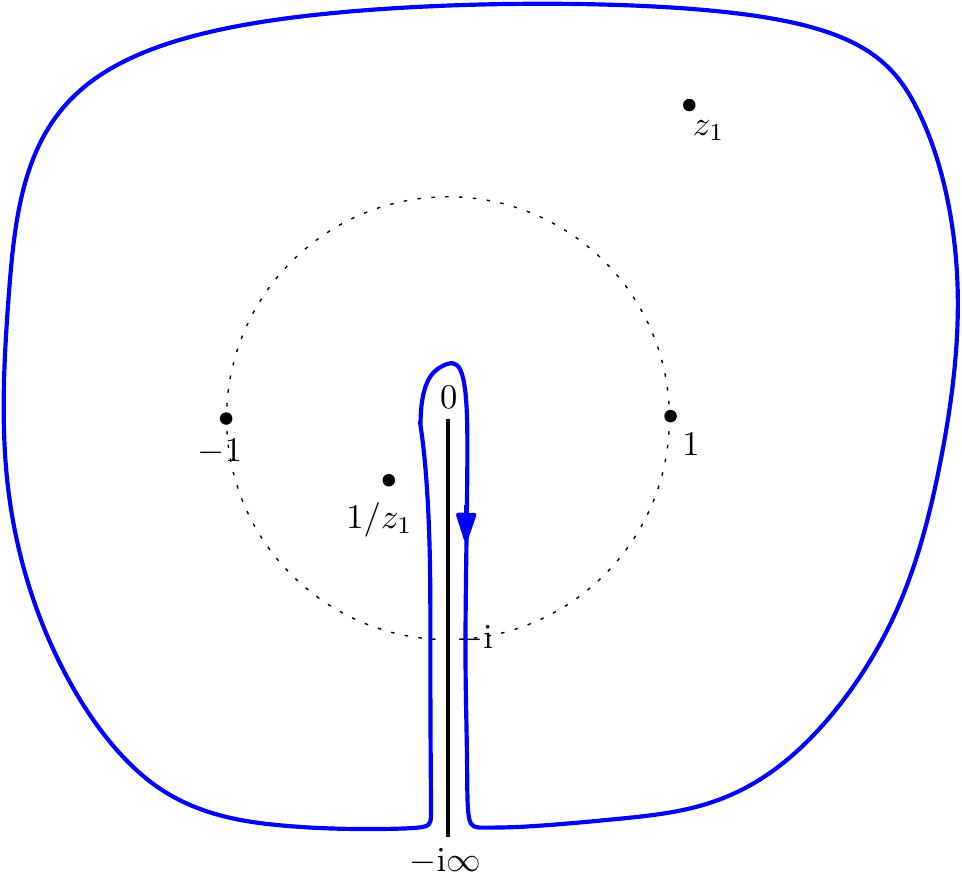}
\caption{The integration contour in $\mathcal{S}$ (the $z$-plane).}\label{Cont1}
\end{figure}

Let us assume $\operatorname{Re}x_1 > 0$ and move the contour (see Fig.~\ref{Cont1}). It will surround the poles at $z = z_1$ and $1/z_1$ -- which give equal contributions and which we handle as in the previous paragraphs~-- as well as the cut of the logarithm for $z$ on the nonpositive imaginary axis, together with a~half-circle arbitrarily close to $0$ and an arbitrarily large circle. When $z$ goes to $\infty$ in~$\mathbb{C}\setminus {\rm i}\mathbb{R}_{-}$ we have $f(z) \in O\big(\dd x(z)/x(z)^2\big)$ since $f$ has no pole at $\infty$. Therefore the integrand in~\eqref{integA} behaves as $O\big(|\dd z|\,|z|^{-3}\ln |z|\big)$ and the large circle pushed to $\infty$ gives a zero contribution. By symmetry $z \mapsto 1/z$ of the integrand the same is true for the contribution of the half-circle pushed to~$0$. The discontinuity of~$\ln z$ on its branchcut, from right to left is then $-2{\rm i}\pi$. Therefore
\begin{align*}
\mathcal{A}[f](x_1) & = -\frac{2f(z_1)}{\dd x_1} \ln\big(x(z_1)/z_1\big) + \frac{2f(z_1)}{\dd x_1} \ln x_1 + \int_{0}^{-{\rm i\infty}} \frac{f(z)}{x_1 - x(z)} \\
& = -\frac{2f(z_1)}{\dd x_1} \ln\big(x(z_1)/z_1\big) + \frac{2f(z_1)}{\dd x_1} \ln x_1 - \left(\int_{0}^{-{\rm i}} + \int_{-{\rm i}}^{-{\rm i}\infty}\right)\frac{f(z)}{x_1 - x(z)}.
\end{align*}
We use $z \mapsto 1/z$ convert the first integral from $0$ to $-{\rm i}$ into a integral from $+{\rm i}\infty$ to ${\rm i}$ in the $z$-plane. This also multiplies $f(z)$ by a minus sign according to~\eqref{linearloopeq}. The resulting integral in the $z$-plane is then equivalent to the integral over the positive imaginary axis in the $x$-plane. The second integral from $-{\rm i}$ to $-{\rm i}\infty$ in the $z$-plane is equivalent to an integral over the negative imaginary axis in the $x$-plane. We therefore obtain
\begin{gather*}
\mathcal{A}[f](x_1) = -\frac{2f(z_1)}{\dd x_1}\,\ln\big(x(z_1)/z_1\big) + \frac{2f(z_1)}{\dd x_1} \ln x_1 - \left(\int_{0}^{+{\rm i}\infty} + \int_{0}^{-{\rm i}\infty}\right)\frac{f(Z(x))}{x_1 - x}.
\end{gather*}
The integrals are closely related to the Hilbert transform, defined for a function $F\colon \mathbb{R}_{\geq 0} \rightarrow \mathbb{R}$ and with $v \in \mathbb{C}\setminus \mathbb{R}_{> 0}$ by
\begin{gather*}
\mathcal{H}[F](v) := \int_{0}^{\infty} \frac{F(u)\,\dd u}{v + u}.
\end{gather*}
Namely, we have with $\check{F}(x) = f(Z(x))/\dd x$
\begin{gather}
\mathcal{A}[f](x_1) = - \frac{2f(z_1)}{\dd x_1} \ln\big(x(z_1)/z_1\big) + \frac{2f(z_1)}{\dd x_1} \ln x_1 + \mathcal{H}[\check{F}({\rm i} \cdot)]({\rm i}x_1) + 2\mathcal{H}[\check{F}(-{\rm i}\cdot)](-{\rm i}x_1).\!\!\!\label{Agngn}
\end{gather}
To obtain the asymptotic expansion of the last terms we can rely on the following result
\begin{lemma}[\cite{WongMcClure}]
\label{HilbertA} Let $F\colon \mathbb{R}_{> 0} \rightarrow \mathbb{C}$ such that, for any integer $b \geq 0$
\begin{gather*}
F(u) = \sum_{a = 0}^{b + 1} \frac{f_{k}}{u^{k + 1}} + r_{b + 1}(u),\qquad {\rm and}\qquad \sup_{u \geq 0} u^{b + 1 + \eta}\,|r_{b + 1}(u)| \leq C_{b}
\end{gather*}
for some $\eta \in (0,1)$ and $C_{b} > 0$. Then for $u \rightarrow \infty$ away from the real axis we have
\begin{gather*}
\mathcal{H}[F](u) = \left(\sum_{a = 0}^{b + 1} \frac{(-1)^{a} f_{a}}{u^{a + 1}}\right)\ln u + \sum_{a = 0}^{b + 1} \frac{(-1)^{a}}{y^{a + 1}}\big({-}H_{a}f_{a} + \mathcal{J}_{a}[F]\big) + O\big(|u|^{-(b + 1 + \alpha)}\big),
\end{gather*}
where $\ln$ is the logarithm with its standard choice of branchcut on the negative real axis, and
\begin{gather*}
\mathcal{J}_{b}[F] := \lim_{\epsilon \rightarrow 0^+} \left(-f_{b}\ln \epsilon + \int_{\epsilon}^{+\infty} (u - \epsilon)^{b} r_{b}(u)\,\dd u\right).
\end{gather*}
\end{lemma}

\begin{corollary}\label{LEPTH}Let $f$ be a meromorphic $1$-form on $\mathcal{S}$ with poles away from $z = 0,\infty$ and such that $f(z) + f(1/z) = 0$. We have the Laurent series expansion when $x_1 \rightarrow \infty$
\begin{gather*}
\sum_{a=\pm 1}\Res_{z =a} \frac{f(z)\,\ln z}{x_1 - x(z)} \sim - 2f(z_1)\ln\big(x(z_1)/z_1\big) + \sum_{b \geq 0} \frac{(b + 1)!}{x_1^{b + 2}}\big(2H_{b + 1}\mathcal{I}_{b}^1[f] + \mathcal{I}_{b + 1}^0[f]\big).
\end{gather*}
where $\mathcal{I}_b^0$ was introduced in the introduction, equation~\eqref{I^0}.
\end{corollary}
\begin{proof}
Let us denote momentarily
\begin{gather*}
 \tilde{f}_{b} = (b + 1)!\,\mathcal{I}_{b}^1[f],\qquad \tilde{r}_{b + 1}(u) = \check{\mathcal{R}}_{b + 1}[f](u),
\end{gather*}
with the convention $\tilde{f}_{-1} = 0$. When we apply Lemma~\ref{HilbertA} with $F(u) = \tilde{F}({\rm i}u)$ for which
\begin{gather*}
f_{b} = (-{\rm i})^{b + 1}\,\tilde{f}_{b - 1},\qquad r_{b + 1}(u) = \tilde{r}_{b + 1}({\rm i}u),\qquad \eta = 1/2,
\end{gather*}
we find
\begin{gather}
\mathcal{H}\big[\tilde{F}({\rm i}\cdot)\big]({\rm i}x_1) + O\big(|x_1|^{-(b + 5/2)}\big)
= \left(\sum_{a = 0}^{b + 1} (-1)^{a}\,(-{\rm i})^{2(a + 1)}\,\frac{\tilde{f}_{a - 1}}{x_1^{a + 1}} \right)\ln({\rm i}x_1) \nonumber\\
 \quad {} + \sum_{a = 0}^{b} \frac{(-1)^{a}(-{\rm i})^{2(a + 1)}}{x_1^{a + 1}}\left\{-H_{a} \tilde{f}_{a - 1} + \lim_{\epsilon \rightarrow 0^+} \left(-\tilde{f}_{a - 1}\ln \epsilon + \int_{\epsilon}^{+\infty} {\rm i}^{a + 1}\,(u - \epsilon)^{a} \tilde{r}_{a}({\rm i}u)\,\dd u\right)\right\} \nonumber\\
 = -\left(\sum_{a = -1}^{b} \frac{\tilde{f}_{a}}{x_{1}^{a + 2}}\right)\ln({\rm i}x_1) \nonumber\\
 \quad {}- \sum_{a = 0}^{b} \frac{1}{x_1^{a + 2}}\left\{-H_{a + 1}\tilde{f}_{a} + \lim_{\epsilon \rightarrow 0^+}\left(-\tilde{f}_{a}\ln \epsilon + \int_{{\rm i}\epsilon}^{+{\rm i}\infty} (x - {\rm i}\epsilon)^{a + 1} \tilde{r}_{a + 1}(x)\,\dd x\right)\right\}.\label{Hilb1}
\end{gather}
In principle, the two sums should start from $a = -1$, but as $\tilde{f}_{-1} = 0$ the first one effectively starts at $a = 0$. In the second one, the $a = -1$ summand only contains the regularised integral. When we apply Lemma~\ref{HilbertA} for $F(u) = \tilde{F}(-{\rm i}u)$, for which
\begin{gather*}
f_{a} = {\rm i}^{a + 1}\,\tilde{f}_{a - 1},\qquad r_{b + 1}(u) = \tilde{r}_{b + 1}(-{\rm i}u),
\end{gather*}
we find
\begin{gather}
\mathcal{H}\big[\tilde{F}(-{\rm i}\cdot)\big](-{\rm i}x_1) + O\big(|x_1|^{-(b + 5/2)}\big)
 = \left(\sum_{a = 0}^{b + 1} (-1)^{a}{\rm i}^{2(a + 1)} \frac{\tilde{f}_{a - 1}}{x_1^{a + 1}}\right)\ln(-{\rm i}x_1) \nonumber\\
 \quad {}+ \sum_{a = 0}^{b} \frac{(-1)^{a}{\rm i}^{2(a + 1)}}{x_1^{a + 1}}\left\{-H_{a}\tilde{f}_{a - 1} + \lim_{\epsilon \rightarrow 0^+} \left(-\tilde{f}_{a - 1}\ln \epsilon + \int_{\epsilon}^{+\infty}\! (-{\rm i})^{a + 1}(u - \epsilon)^{a} \tilde{r}_{a}(-{\rm i}u)\,\dd u\right)\right\} \nonumber\\
 = - \left(\sum_{a = 0}^{b} \frac{\tilde{f}_{a}}{x_1^{a + 2}}\right)\ln(-{\rm i}x_1) \nonumber\\
 \quad {}- \sum_{a = -1}^{b} \frac{1}{x_1^{a + 2}}\left\{-H_{a + 1}\tilde{f}_{a} + \lim_{\epsilon \rightarrow 0^+} \left(-\tilde{f}_{a}\ln\epsilon + \int_{-{\rm i}\epsilon}^{-{\rm i}\infty} (x + {\rm i}\epsilon)^{a + 1}\,\tilde{r}_{a+ 1}(x)\,\dd x\right)\right\}.\label{Hilb2}
\end{gather}
We multiply \eqref{Hilb1} and \eqref{Hilb2} by $2$ and sum them, in view of obtaining the asymptotic expansion of \eqref{Agngn}. We observe that the logarithm term in \eqref{Hilb1}--\eqref{Hilb2} contributes to
\begin{gather*}
-2\left(\sum_{a= 0}^{b}\frac{\tilde{f}_{a}}{x_1^{a + 2}}\right)\ln x_1,
\end{gather*}
and therefore cancels the second term in \eqref{Agngn}. The final result is
\begin{gather*}
 \mathcal{A}[f](x_1) = -\frac{2f(z_1)}{\dd x_1} \ln\big(x(z_1)/z_1\big) + O\big(|x_1|^{-(b + 5/2)}\big) + \sum_{a = 0}^b \frac{1}{x_1^{a + 2}}\bigg\{2H_{a + 1}\tilde{f}_{a} \\
\hphantom{\mathcal{A}[f](x_1) =}{} + \lim_{\epsilon \rightarrow 0^+}\left(2\tilde{f}_{a}\ln \epsilon - 2\int_{0}^{{\rm i}\infty} x^{a + 1}\tilde{r}_{a + 1}(x + {\rm i}\epsilon)\,\dd x - \int_{0}^{-{\rm i}\infty} x^{a + 1}\,\tilde{r}_{a + 1}(x - {\rm i}\epsilon)\,\dd x\right)\bigg\} \\
\hphantom{\mathcal{A}[f](x_1)}{} = -\frac{2f(z_1)}{\dd x_1} \ln\big(x(z_1)/z_1\big) + \sum_{a = 0}^b \frac{(a + 1)!}{x_1^{a + 2}}\big(2H_{a + 1}\mathcal{I}_{a}^1 + \mathcal{I}_{a + 1}^0\big)[f]
\end{gather*}
in terms of the functional $\mathcal{I}^0$ introduced in \eqref{I^0}.
\end{proof}

\subsection[Decomposition of $\protect{\mathcal{L}[f]}$]{Decomposition of $\boldsymbol{\mathcal{L}[f]}$}

Recall the basis $\xi_m^{\alpha}$ defined in~\eqref{auxdif}. The following lemma gives a decomposition of $\cl\big(\xi_m^\alpha\big)(z_1,z_2)$, defined in~\eqref{Ldefdede}. which in particular implies its polar behaviour in $z_2$. It will be applied in Section~\ref{S7}.

\begin{lemma} For any $m \geq 0$ and $\alpha \in \{0,1\}$, we have \label{Ltransform}
\begin{gather*}
\cl\big(\xi_m^\alpha\big)(z_1,z_2)=\sum_{\substack{\beta = 0,1 \\ \ell \geq 0}} c^{\alpha,\beta}_{m,\ell}(x_1)\xi_{\ell}^\beta(z_2),
\end{gather*}
where $\xi_{\ell}^\beta(z_2)$ is defined in \eqref{auxdif} and $c^{\alpha,\beta}_{m,\ell}(x_1)=\big(x_1^2-4\big)^{-N} P^{\alpha,\beta}_{m,\ell,N}(x_1)\dd x_1$ for some $N\in\bn$ and $P^{\alpha,\beta}_{m,\ell,N}(x_1)$ is a polynomial of degree at most $2N-2$.
\end{lemma}
\begin{proof}The differentials $c^{\alpha,\beta}_{m,\ell}(x_1)$ and $\xi_{\ell}^\beta(z_2)$ form bases of the space of meromorphic differentials satisfying the following properties:
\begin{itemize}\itemsep=0pt
\item[$(i)$] $\cl\big(\xi_m^\alpha\big)(z_1,1/z_2)=-\cl\big(\xi_m^\alpha\big)(z_1,z_2)$;
\item[$(ii)$] $\cl\big(\xi_m^\alpha\big)(1/z_1,z_2)=\cl\big(\xi_m^\alpha\big)(z_1,z_2)$;
\item[$(iii)$] $\cl\big(\xi_m^\alpha\big)(z_1,z_2)$ is meromorphic in $z_1$ with poles only at $z_1=\pm 1$;
\item[$(iv)$] $\cl\big(\xi_m^\alpha\big)(z_1,z_2)$ is meromorphic in $z_2$ with poles only at $z_2=\pm 1$;
\item[$(v)$] for any $a \in \{-1,1\}$, $\Res_{z_2= a} \cl\big(\xi_m^\alpha\big)(z_1,z_2)=0$.
\end{itemize}
Hence it is enough to prove that $\cl\big(\xi_m^\alpha\big)(z_1,z_2)$ satisfies these properties.

$(i)$ follows from $\xi_m^\alpha(1/z_2)=-\xi_m^\alpha(z_2)$, the symmetry $x_2=x(z_2)=x(1/z_2)$, and the oddness of $\omega_{0,2}^{{\rm odd}}(z_1,z_2)$ under $z_2\mapsto1/z_2$. Property $(ii)$ follows in a similar way once we use $\omega_{0,2}(z_1,1/z_2)=\omega_{0,2}(1/z_1,z_2)$ and the oddness of $\xi_m^\alpha(z_1)$ to get a symmetric bidifferential.

For $(iii)$, clearly $\cl\big(\xi_m^\alpha\big)(z_1,z_2)$ is meromorphic in $z_1$ with poles at $z_1=1,-1,z_2,1/z_2$ so we need to show that the poles at $z_1=z_2$ and $z_1=1/z_2$ are removable. In fact, we only need show that the pole at $z_1=z_2$ is removable and $(i)$ will imply the same at $z_1=1/z_2$. The pole on the diagonal $z_1=z_2$ has order 2, so consider
\begin{align*}
\lim_{z_1\to z_2} \frac{(z_1-z_2)^2}{\dd z_1\,\dd z_2}\cl\big(\xi_m^\alpha\big)(z_1,z_2) & = \lim_{z_1\to z_2} \frac{(z_1-z_2)^2}{\dd z_1\,\dd z_2}\!\left\{\omega_{0,2}(z_1,z_2)\frac{\xi_m^\alpha(z_1)}{\dd x_1} - \frac{\xi_m^\alpha(z_2)}{\dd x_2}\frac{\dd x_1\dd x_2}{(x_1-x_2)^2} \right\}\\
 & = \frac{\xi_m^\alpha(z_2)}{\dd x_2} - \frac{\xi_m^\alpha(z_2)}{\dd x_2}\frac{\dd x_2^2}{\dd z_2^2\,x'(z_2)^2} \\
 & = 0.
\end{align*}
where the first equality removed those terms of $\cl\big(\xi_m^\alpha\big)(z_1,z_2)$ with a simple pole (and possibly holomorphic) at $z_1=z_2$. Hence the pole at $z_1=z_2$ is at most simple and we shall compute its residue. The residue of $\cl\big(\xi_m^\alpha\big)(z_1,z_2)$ at $z_2=z_1$ gives the same residue and is simpler to calculate. In fact it is immediately 0 because $\Res\limits_{z_2=z_1}\omega_{0,2}(z_1,z_2)=0$ and $\omega_{0,2}(z_1,1/z_2)$ has no pole at $z_2=z_1$, and the final term in $\cl\big(\xi_m^\alpha\big)(z_1,z_2)$ is exact in $z_2$ so all residues vanish. Hence the pole is removable at $z_1=z_2$. This discussion also implies~$(iv)$.

Finally, property $(v)$ follows from property $(i)$ and the fact that $z_1=\pm 1$ are the fixed points of the involution $z_1\mapsto 1/z_1$.
\end{proof}

The main application of Lemma~\ref{Ltransform} is to show that the operators $\ci^{\alpha_i}_{b_i}$ commute on \eqref{diffdecay} as explained in the proof of Theorem~\ref{th:vir} below. Lemma~\ref{Ltransform} also shows us that evaluation of $\bigotimes_{i=1}^n\ci^{\alpha_i}_{b_i}$ on \eqref{diffdecay} depends only on the values of the regularised integral $\ci^0_j$ applied to the differentials $\xi_m^\alpha$ which are determined by $\ci^0_j\big(\xi_m^\alpha\big)=\ci^0_{j-m}(\xi_0^\alpha)$ and the table of Proposition~\ref{PROPJIU}. This is because the right-hand side of \eqref{diffdecay} is a linear combination of the differentials $\xi_m^\alpha(z_j)$ (with coefficients given by differentials in the other variables), i.e., it has poles only at $z_j=\pm1$ for $j \in \{2,\ldots,n\}$~-- the other poles are removable~-- and is odd under $z_j\mapsto 1/z_j$.

\section{Properties of regularised contour integrals}\label{regint}

If $f$ is a meromorphic $1$-form on $\mathcal{S}$ without poles at $z = 0$ and $\infty$, we can define
\begin{gather*}
\mathcal{R}_{b}[f] := f - \sum_{a = 0}^{b - 1} \frac{f_{a}}{x^{a + 2}},\qquad f_a = -\Res_{z = \infty} x^{a + 1}\,f,
\end{gather*}
and $\check{\mathcal{R}}_{b}[f](x) = \mathcal{R}_{b}[f](Z(x))$. Then, if $f$ has no pole for $z \in {\rm i}\mathbb{R}$, we can define for $b \geq 0$
\begin{gather}
\label{Iooooo} \mathcal{I}_{b}^0[f] := \lim_{\epsilon \rightarrow 0^+} \left(2\mathcal{I}_{b - 1}^{1}[f]\,\ln \epsilon - \int_{0}^{{\rm i}\infty} \frac{x^{b}}{b!} \check{\mathcal{R}}_{b}[f](x + {\rm i}\epsilon)\,\dd x - \int_{0}^{-{\rm i}\infty} \frac{x^{b}}{b!} \check{\mathcal{R}}_{b}[f](x - {\rm i}\epsilon)\,\dd x\right).\!\!\!
\end{gather}

\subsection{Basic properties and extended definition} \label{propint}

We justify that the regularised integral is actually an integral when applied to $1$-forms that are odd with respect to the involution and that do not need regularisation.
\begin{lemma}
\label{IntL} If $f$ is a meromorphic $1$-form on $\mathcal{S}$ with no poles at $z = \pm {\rm i}$, without residues, such that $f(z) + f(1/z) = 0$ and $\mathcal{I}_{a}^1[f] = 0$ for all $a \in \{0,\ldots,b - 1\}$, then
\begin{gather*}
\mathcal{I}_{b}^0[f] = -\int_{\gamma} \frac{x^{b}}{b!}\,f,
\end{gather*}
where $\gamma$ is the contour from $z = 0$ to $z = {\rm i}\infty$.
\end{lemma}
\begin{proof}
The conditions on $f$ imply that the integral in the right-hand side is well-defined, and does not depend on the choice of contour from $0$ to ${\rm i}\infty$. It also shows that
\begin{gather*}
\mathcal{I}_{b}^0[f] = -\int_{{\rm i}}^{{\rm i}\infty} \frac{x(z)^{b}}{b!} f(z) - \int_{-{\rm i}}^{-{\rm i}\infty} \frac{x(z)^{b}}{b!} f(z),
\end{gather*}
where the contour from ${\rm i}$ to ${\rm i}\infty$ is avoiding the (finitely many) poles of $f$, and the result does not depend on such a choice of contour.
We transform the second integral using the change of variable $z \mapsto 1/z$ and the oddness of $f$ with respect to this involution
\begin{gather*}
\mathcal{I}_{b}^0[f] = -\int_{{\rm i}}^{{\rm i}\infty} \frac{x(z)^{b}}{b!} f(z) - \int_{0}^{{\rm i}} \frac{x(z)^{b}}{b!} f(z) = -\int_{0}^{{\rm i}\infty} \frac{x(z)^{b}}{b!} f(z).\tag*{\qed}
\end{gather*}\renewcommand{\qed}{}
\end{proof}

Notice that if $f$ is a meromorphic $1$-form on $\mathcal{S}$, then integration by parts yields
\begin{gather}
\label{IPP222} \mathcal{I}_{b}^1[\dd(f/\dd x)] = -\mathcal{I}_{b - 1}^1[f].
\end{gather}
We now prove a similar property for the regularised integral.
\begin{lemma}
\label{IPP} If $f$ is a meromorphic $1$-form on $\bp^1$ with poles away from $z = \pm {\rm i},0,\infty$, then
\begin{gather*}
\forall\, b \geq 1,\qquad \mathcal{I}_{b}^0[\dd(f/\dd x)] = -\mathcal{I}_{b - 1}^0[f].
\end{gather*}
\end{lemma}
\begin{proof} Let $\tilde{f} = \dd(f/\dd x)$. We observe that for any $b \geq 1$
\begin{gather*}
\mathcal{I}_{b}^1[\tilde{f}] = -\mathcal{I}_{b - 1}^1[f],\qquad \check{\mathcal{R}}_{b}[\tilde{f}](x) = \partial_{x}\,\check{\mathcal{R}}_{b - 1}[f](x).
\end{gather*}
Therefore
\begin{gather*}
-\mathcal{I}_{b}^0[\tilde{f}] = \lim_{\epsilon \rightarrow 0^+}\bigg(2 \mathcal{I}_{b - 2}^1[f] \ln\epsilon + \int_{0}^{{\rm i}\infty} \frac{x^{b}}{b!}\,\partial_{x} \check{\mathcal{R}}_{b - 1}[f](x + {\rm i}\epsilon)\,\dd x \\
\hphantom{-\mathcal{I}_{b}^0[\tilde{f}] =}{} + \int_{0}^{-{\rm i}\infty} \frac{x^{b}}{b!}\,\partial_{x}\mathcal{R}_{b - 1}[f](x - {\rm i}\epsilon)\,\dd x\bigg).
\end{gather*}
An integration by parts yields
\begin{gather*}
-\mathcal{I}_{b}^0[\tilde{f}] = \lim_{\epsilon \rightarrow 0^+} \bigg(2\mathcal{I}_{b - 2}^1[f]\,\ln\epsilon - \int_{0}^{{\rm i}\infty} \frac{x^{b - 1}}{(b - 1)!}\,\check{\mathcal{R}}_{b - 1}[f](x + {\rm i}\epsilon)\dd x\\
\hphantom{-\mathcal{I}_{b}^0[\tilde{f}] =}{} - \int_{0}^{-{\rm i}\infty} \frac{x^{b - 1}}{(b - 1)!}\,\check{\mathcal{R}}_{b - 1}[f](x - {\rm i}\epsilon)\dd x \\
\hphantom{-\mathcal{I}_{b}^0[\tilde{f}] =}{} + \left[\frac{x^{b}}{b!}\,\check{\mathcal{R}}_{b - 1}[f](x + {\rm i}\epsilon)\right]_{0}^{{\rm i}\infty} + \left[\frac{x^{b}}{b!}\,\check{\mathcal{R}}_{b - 1}[f](x - {\rm i}\epsilon)\right]_{0}^{-{\rm i}\infty}\bigg).
\end{gather*}
The first line is by definition $\mathcal{I}_{b - 1}^0[f]$. Since $\check{\mathcal{R}}_{b - 1}[f](x) \in O\big(|x|^{-(b + 1)}\,|\dd x|\big)$ when $|x| \rightarrow \infty$, the boundary terms $\pm {\rm i}\infty$ in the last line do not contribute. And the boundary terms at $0$ vanish because of the power of $x$ in prefactor and the fact that $\epsilon > 0$ before we take the limit.
\end{proof}

\begin{remark} In \eqref{Iooooo} we gave a definition of the linear operator $\mathcal{I}_{b}^0$ for $1$-forms in $\mathcal{S}$ having no poles above $x = \infty$. We shall extend this definition, whenever necessary if $f$ has poles at $\infty$ or other singularities, such that the integration by parts (Lemma~\ref{IPP}) and linearity continue to hold.
\end{remark}

\subsection{Evaluation on the odd basis}\label{S53}
We now evaluate $\mathcal{I}_{b}^{\alpha}$ on the basis $\{\xi^{\alpha}_k(z)\,|\, k\in\bn,\, \alpha\in\{0,1\}\}$, defined in~\eqref{auxdif}, of meromorphic $1$-forms on $\bp^1$ with poles at $z = \pm 1$ and that are odd under $z \mapsto 1/z$.
By \eqref{IPP222} and Lemma~\ref{IPP} we have
\begin{gather*}
\mathcal{I}_{a}^{\alpha}\big[\xi_{b}^{\beta}\big] = \mathcal{I}_{a - b}^{\alpha}\big[\xi^{\beta}_{0}\big],
\end{gather*}
so it is enough to evaluate the operators on $\xi^{\beta}_{0}$ for $\beta = 0,1$.

\begin{proposition}\label{PROPJIU}The operators $\mathcal{I}_{a}^{\alpha}$ evaluate on the basis as follows:
\begin{gather*}
\renewcommand{\arraystretch}{1.6}
\begin{array}{|c||c|c||c|c|}
\hline
\phantom{xxxx} & \mathcal{I}^0_{2m} & \mathcal{I}^0_{2m + 1} & \mathcal{I}^1_{2m} & \mathcal{I}^1_{2m + 1}\bsep{4pt} \\
\hline
\xi^0_{0} & \dfrac{1 - 2mH_{m}}{m!^2} & 0 & 0 & \dfrac{1}{m!(m + 1)!} \tsep{8pt}\bsep{8pt}\\
\hline
\xi^1_0 & 0 & -\dfrac{2H_{m}}{m!^2} & \dfrac{1}{m!^2} & 0 \tsep{8pt}\bsep{8pt}\\
\hline
\end{array}
\renewcommand{\arraystretch}{1.0}
\end{gather*}
\end{proposition}
\begin{proof} The evaluation of $\mathcal{I}^1$ is determined by the Laurent series expansion of $\xi_{0}^{\beta}$. We have
\begin{gather*}
\frac{\xi_{0}^0}{\dd x} = \frac{2}{(x^2 - 4)^{3/2}} \sim \sum_{m \geq 0} \frac{(2m + 2)!}{m!(m + 1)!}\,\frac{1}{x^{2m + 3}}, \\
\frac{\xi_{0}^1}{\dd x} = \frac{x}{(x^2 - 4)^{3/2}} \sim \sum_{m \geq 0} \frac{1}{2}\,\frac{(2m + 2)!}{m!(m + 1)!}\,\frac{1}{x^{2m + 2}},
\end{gather*}
which yield the entries of the last two columns. The evaluation of $\mathcal{I}^0$ can be computed via Lemma~\ref{LEPTH}. Let us introduce the formal series
\begin{gather*}
L(t) := \frac{1}{(1 - 4t)^{3/2}}\ln\left(\frac{2}{1 + \sqrt{1 - 4t}}\right) := \sum_{m \geq 0} L_m\,t^{m}.
\end{gather*}
We will compute the $L_m$ more explicitly in Lemma~\ref{Lemlm} at the end of the proof. For $f = \xi^{0}_{0}$ we compute
\begin{gather*}
\sum_{a=\pm1}\Res_{z =a} \frac{\xi_0^0(z)\ln z}{x_1 - x(z)} = \frac{x_1}{x_1^2 - 4} \sim \sum_{m \geq -1} \frac{2^{2m + 2}}{x_1^{2m + 3}}.
\end{gather*}
Note that the choice of determination of the logarithm (provided it is holomorphic in a neighborhood of $1$ and $-1$) does not affect the result. We also have
\begin{gather*}
2f(z_1)\ln(x(z_1)/z_1) = \frac{4L(x_1^{-2})}{x_1^3} \sim \sum_{m \geq 0} \frac{4L_{m}}{x^{2m + 3}}.
\end{gather*}
Using the values of $\mathcal{I}^1[\xi_{0}^{0}]$ already found, we deduce from Lemma~\ref{LEPTH}
\begin{gather}
\label{LANI}\mathcal{I}_{2m + 1}\big[\xi_0^0\big] = 0,\qquad \mathcal{I}_{2m}\big[\xi_{0}^{0}\big] = \frac{2^{2m} + 4L_{m - 1}}{2m!} - \frac{2H_{2m}}{(m - 1)!m!}.
\end{gather}
For $f = \xi^{1}_{0}$ we compute
\begin{gather*}
\sum_{a=\pm1}\Res_{z =a} \frac{\xi_0^1(z)\,\ln z}{x_1 - x(z)} = \frac{2}{x_1^2 - 4} \sim \sum_{m \geq 0} \frac{2^{2m + 1}}{x_1^{2m + 2}},
\end{gather*}
and
\begin{gather*}
2f(z_1)\ln(x(z_1)/z_1) = \frac{2L\big(x_1^{-2}\big)}{x_1^2} \sim \sum_{m \geq 0} \frac{2L_{m}}{x^{2m + 2}}.
\end{gather*}
Using the known values of $\mathcal{I}^{1}$ on $\xi_0^1$ we get
\begin{gather}\label{LANJ}\mathcal{I}_{2m + 1}\big[\xi^1_{0}\big] = \frac{2\big(2^{2m} + L_m\big)}{(2m + 1)!} - \frac{2H_{2m + 1}}{m!^2},\qquad \mathcal{I}_{2m}\big[\xi_{0}^{1}\big] = 0.
\end{gather}

Now let us evaluate the constants $L_m$.

\begin{lemma}\label{Lemlm} For any $m \geq 1$, we have
\begin{gather*}
L_m = -2^{2m} + \frac{(2m + 1)!}{m!^2}(H_{2m + 1} - H_{m + 1}) .
\end{gather*}
\end{lemma}

\begin{proof}With the change of variable $t = \frac{v}{(1 + v)^2}$, we compute
\begin{align*}
L_m & = \Res_{t = 0} \frac{\dd t}{t^{m + 1}\,(1 - 4t)^{3/2}}\,\ln\left(\frac{2}{1 + \sqrt{1 - 4t}}\right) \\
& = \Res_{v = 0} \frac{\dd v\,(1 + v)^{2m + 2}\,\ln(1 + v)}{v^{m + 1}(1 - v)^2} \\
& = \frac{\dd}{\dd \epsilon}\left(\Res_{v = 0} \frac{\dd v\,(1 + v)^{2m + 2 + \epsilon}}{v^{m + 1}(1 - v)^2}\right)\bigg|_{\epsilon = 0} \\
& = \frac{\dd}{\dd\epsilon}\left(\sum_{a = 0}^{m} \frac{(a + 1)\Gamma(2m + 3 + \epsilon)}{(m - a)!\Gamma(m + 3 + \epsilon + a)!}\right)\bigg|_{\epsilon = 0}.
\end{align*}
Using that $(\ln \Gamma)'(k + 1) = -\gamma_{E} + H_{k}$ for any positive integer $k$ where $\gamma_{E}$ is the Euler--Mascheroni constant, we deduce
\begin{gather}
\label{Lmexp} L_m = \sum_{a = 0}^{m} \frac{(a + 1)\,(2m + 2)!}{(m - a)!(m + 2 + a)!}\big(H_{2m + 2} - H_{m + 2 + a}\big) = K_m\big(H_{2m + 2} - H_{m + 1}\big) + \Delta_m,
\end{gather}
where
\begin{gather*}
K_m := \sum_{a = 0}^{m} \frac{(a + 1)\,(2m + 2)!}{(m - a)!(m + 2 + a)!},\qquad \Delta_m := \sum_{a = 0}^{m} \frac{(a + 1)\,(2m + 2)!}{(m - a)!(m + 2 + a)!} \sum_{j = m + 2}^{m + 2 + a} \frac{1}{j}.
\end{gather*}
These two sums can be computed in an elementary way. Let us introduce the auxiliary sum for $c \in \{0,1,\ldots,m\}$
\begin{gather*}
\tilde{K}_{m,c} = \sum_{b = 0}^{c} \frac{(m + 1 - b)\,(2m + 2)!}{b!(2m + 2 - b)!}.
\end{gather*}
An easy induction on $c$ shows that
\begin{gather*}
\tilde{K}_{m,c} = \frac{1}{2}\,\frac{(2m + 2)!}{c!(2m + 1 - c)!}.
\end{gather*}
The change of index $a = m - b$ shows that $K_m = \tilde{K}_{m,m}$, hence
\begin{gather}
\label{Kmexp} K_m = \frac{1}{2}\,\frac{(2m + 2)!}{m!(m + 1)!} = \frac{(2m + 1)!}{m!^2}.
\end{gather}
It remains to evaluate
\begin{align*}
\Delta_m & = \sum_{j = m + 2}^{2m + 2} \frac{1}{j} \sum_{a = j - (m + 2)}^{m} \frac{(a + 1)\,(2m + 2)!}{(m - a)!(m + 2 + a)!} \\
& = \sum_{j = m + 2}^{2m + 2} \frac{1}{j} \sum_{b = 0}^{2m + 2 - j} \frac{(m + 1 - b)\,(2m + 2)!}{b!(2m + 2 - b)!} \\
& = \sum_{j = m + 2}^{2m + 2} \frac{1}{j} \tilde{K}_{m,2m + 2 - j}.
\end{align*}
Therefore
\begin{align}
\Delta_m & = \frac{1}{2} \sum_{j = m + 2}^{2m + 2} \frac{(2m + 2)!}{(2m + 2 - j)!(j - 1)!} = \frac{1}{2} \sum_{j = m + 2}^{2m + 2} \frac{(2m + 2)!}{(2m + 2 - j)!j!} \nonumber\\
 & = \frac{1}{4} \sum_{j = 0}^{2m + 2} \frac{(2m + 2)!}{(2m + 2 - j)!j!} \nonumber\\
 & = \frac{1}{4} 2^{2m + 2} = 2^{2m}.\label{Deltamexp}
\end{align}
Inserting \eqref{Kmexp} and \eqref{Deltamexp} in \eqref{Lmexp} establishes our formula for $L_m$.
\end{proof}

Inserting this result into \eqref{LANI} and \eqref{LANJ} yields the two first columns and concludes the proof of Proposition~\ref{PROPJIU}.
\end{proof}

\begin{proposition} \label{xaction}
For $j,k,m\in\bn$ and $\alpha,\beta\in\{0,1\}$
\begin{gather} \label{opcomp}
\ci^\beta_j\left(\frac{x^k}{k!}\xi_m^\alpha\right) = \binom{j+k+\beta}{k}\,\ci^\beta_{j+k}\big(\xi_m^\alpha\big) + 2\delta_{0,\beta}\binom{j+k}{k}(H_{j + k} - H_j)\,\ci^1_{j+k-1}\big(\xi_m^\alpha\big).
\end{gather}
\end{proposition}
\begin{proof}
This is straightforward when $\beta=1$, coming from $\frac{1}{(j+1)!k!}=\binom{j+k+1}{k}\frac{1}{(j+k+1)!}$.
The main content of the lemma is the case $\beta=0$ which we prove by induction on~$k$.

When $k=0$, the second term vanishes and \eqref{opcomp} is true in this case. The inductive argument requires the identity
\begin{gather} \label{xact}
x\xi_m^\alpha=2\xi_m^{1-\alpha}+(m+\alpha)\xi_{m-1}^{\alpha},\qquad m\geq 0,
\end{gather}
which is proven by induction on $m$ by applying $-\dd(\cdot/\dd x)$ to both sides of \eqref{xact}. The initial case $m=0$ of~\eqref{xact} is an explicit calculation for $\alpha=0$ and $\alpha=1$ involving $\xi^\alpha_{-1}(z)$ defined in~\eqref{auxdif}.

Given $k>0$, assume \eqref{opcomp} is true for $k-1$. Then
\newcommand{\jk}{j\hspace{-.5mm}+\hspace{-.5mm}k\hspace{-.5mm}-\hspace{-.5mm}1}
\begin{gather*}
 \ci^0_j\left(\frac{x^k}{k!}\xi_m^\alpha\right) = \ci^0_j\left(\frac{x^{k-1}}{k!}x\xi_m^\alpha\right)=\ci^0_j\left(\frac{x^{k-1}}{k!}\big[2\xi_m^{1-\alpha}+(m+\alpha)\xi_{m-1}^{\alpha}\big]\right)\\
 \hphantom{\ci^0_j\left(\frac{x^k}{k!}\xi_m^\alpha\right)}{} = \frac{2}{k}\ci^0_j\left(\frac{x^{k-1}}{(k-1)!}\xi_m^{1-\alpha}\right)+\frac{m+\alpha}{k}\ci^0_j\left(\frac{x^{k-1}}{(k-1)!}\xi_{m-1}^{\alpha}\right)\\
\hphantom{\ci^0_j\left(\frac{x^k}{k!}\xi_m^\alpha\right)}{} = \frac{2}{k}\binom{\jk}{k-1} \ci^0_{j+k-1}\big(\xi_m^{1-\alpha}\big) + \frac{4}{k}\binom{\jk}{k-1}(H_{j + k - 1} - H_{j})\,\ci^1_{j+k-2}\big(\xi_m^{1-\alpha}\big)\\
\hphantom{\ci^0_j\left(\frac{x^k}{k!}\xi_m^\alpha\right)=}{}+\frac{m+\alpha}{k}\binom{\jk}{k-1} \ci^0_{j+k-1}\big(\xi_{m-1}^{\alpha}\big) \\
\hphantom{\ci^0_j\left(\frac{x^k}{k!}\xi_m^\alpha\right)=}{}+ \frac{2(m+\alpha)}{k}\binom{\jk}{k-1}(H_{j + k - 1} - H_{j})\,\ci^1_{j+k-2}\big(\xi_{m-1}^\alpha\big),
\end{gather*}
where the second equality uses \eqref{xact} and the final equality uses the inductive hypothesis. We manipulate this expression for $\ci^0_j\big(\frac{x^k}{k!}\xi_m^\alpha\big)$ to consist of only evaluations involving $\xi_m^{\alpha}$ as follows. We use integration by parts $\ci^\alpha_{i-1}\big(\xi_{m-1}^{\alpha}\big)=\ci^\alpha_i\big(\xi_{m}^{\alpha}\big)$ and for those evaluations involving $\xi_m^{1-\alpha}$ substitute
\begin{gather*}
\ci^\beta_{j+k-1-\beta}\big(\xi_m^{1-\alpha}\big)=\tfrac12(j+k-m-\alpha)\ci^\beta_{j+k-\beta}\big(\xi_m^{\alpha}\big)+\delta_{\beta,0}\ci^1_{j+k-1}\big(\xi_m^{\alpha}\big),
\end{gather*}
which can be checked using the table of values in Proposition~\ref{PROPJIU}. Collecting the coefficients of $\mathcal{I}_{j + k}^{0}\big(\xi_{m}^{\alpha}\big)$ and $\mathcal{I}_{j + k - 1}^1\big(\xi_{m}^{\alpha}\big)$ we get
\begin{gather*}
\ci^0_j\left(\frac{x^k}{k!}\xi_m^\alpha\right) = \frac{2}{k}\binom{\jk}{k-1}(H_{j + k - 1} - H_{j})(j+k-m-\alpha) \ci^1_{j+k-1}\big(\xi_{m}^{\alpha}\big) \\
 \hphantom{\ci^0_j\left(\frac{x^k}{k!}\xi_m^\alpha\right) =}{} + \frac{2(m+\alpha)}{k}\binom{\jk}{k-1}(H_{j + k - 1} - H_{j}) \ci^1_{j+k-1}(\xi_{m}^\alpha) \\
\hphantom{\ci^0_j\left(\frac{x^k}{k!}\xi_m^\alpha\right)}{}= \frac{j+k}{k}\binom{\jk}{k-1} \ci^0_{j+k}\big(\xi_m^\alpha\big) \\
\hphantom{\ci^0_j\left(\frac{x^k}{k!}\xi_m^\alpha\right) =}{}+ 2 \binom{\jk}{k-1}\left(\frac{j + k}{k}(H_{j + k - 1} - H_{j}) + \frac{1}{k}\right) \ci^1_{j+k-1}\big(\xi_m^\alpha\big) \\
\hphantom{\ci^0_j\left(\frac{x^k}{k!}\xi_m^\alpha\right)}{}= \binom{j+k}{k}\big\{\ci^0_{j+k}\big(\xi_m^\alpha\big) + 2(H_{j + k} - H_{j}) \ci^1_{j+k-1}\big(\xi_m^\alpha\big)\big\}.
\end{gather*}
Hence \eqref{opcomp} for $k-1$ implies \eqref{opcomp} for $k$, and by induction \eqref{opcomp} is true for $k\geq 0$.
\end{proof}

\subsection[Evaluation of one operator on $\omega_{0,2}^{{\rm odd}}$]{Evaluation of one operator on $\boldsymbol{\omega_{0,2}^{{\rm odd}}}$}

In the statement of Theorem~\ref{insertions} for $(0,2)$, we need to consider $\mathcal{I}_{k}^{\alpha} \otimes \mathcal{I}_{\ell}^{\beta}\big[\omega_{0,2}^{{\rm odd}}\big]$. We can certainly pose
\begin{gather*}
\eta_{k}^{\alpha}(z_0) := \mathcal{I}_{k}^{\alpha}\big[\omega_{0,2}^{{\rm odd}}(\cdot,z_0)\big],
\end{gather*}
but it is not obvious that we can then apply $\mathcal{I}_{\ell}^{\beta}$. Indeed, we now show that $\eta_{k}^{\alpha}(z)$ has singularities, but they are such that after sufficiently many integration by parts it will become a meromorphic $1$-form without singularities at $\infty$.

\begin{lemma}\label{etade} For $k \geq 0$ we have
\begin{gather*}
\eta_k^1(z_0) = \dd_{z_0}\Bigg(\sum_{\substack{b,c \geq 0 \\ 2b + c = k}} \frac{1}{b!(k + 1 - b)!} \frac{z_0^{c + 1} - z_0^{-(c + 1)}}{2}\Bigg), \\
\eta_k^0(z_0) = \dd_{z_0}\Bigg(\frac{x(z_0)^{k}}{k!}\,\tfrac{1}{2}\big(\ln z_0 - \ln\big(z_0^{-1}\big)\big) - \sum_{\substack{b,c \geq 0 \\ 2b + c = k - 1}} \frac{H_{k - b}}{b!(k - b)!}\big(z_0^{c + 1} - z_0^{-(c + 1)}\big)\Bigg).
\end{gather*}
\end{lemma}
\begin{proof} For $\alpha = 1$ we have
\begin{equation*}
\begin{split}
\eta_k^1(z_0) & = \frac{1}{2}\bigg( \Res_{z = 0} \frac{x(z)^{k + 1}}{(k + 1)!}\,\frac{\dd z\,\dd z_0}{(1 - zz_0)^2} - (z_0 \rightarrow 1/z_0)\bigg) \\
& = \frac{1}{2}\bigg(\sum_{c \geq 0} (c + 1)z_0^{c} \big[z^{k - c}\big] \frac{\big(1 + z^2\big)^{k + 1}}{(k + 1)!} \,\dd z_0 - (z_0 \rightarrow 1/z_0)\bigg) \\
& = \dd_{z_0}\bigg(\sum_{\substack{b,c \geq 0 \\ 2b + c = k}} \frac{1}{b!(k + 1 - b)!}\,\frac{z_0^{c + 1} - z_0^{-(c + 1)}}{2}\bigg).
\end{split}
\end{equation*}
For $\alpha = 0$, we apply Lemma~\ref{LEPTH}
\begin{gather*}
 \eta_k^0(z_0) = \frac{1}{2}\bigg\{ \Res_{z = \infty} \frac{x(z)^{k}}{k!}\left( \frac{2\ln\big(1 + z^{-2}\big) \dd z\,\dd z_0}{(z - z_0)^2}+ \Res_{\tilde{z} = z_0} \frac{\ln \tilde{z}\,\dd \tilde{z}\,\dd z_0}{(\tilde{z} - z_0)^2}\,\frac{\dd x(z)}{x(z) - x(\tilde{z})}\right) \\
\hphantom{\eta_k^0(z_0) =}{} +2H_{k} \Res_{z = \infty} \frac{x^{k}(z)}{k!}\,\frac{\dd z\,\dd z_0}{(z - z_0)^2} \quad - (z_0 \rightarrow 1/z_0)\bigg\}\\
\hphantom{\eta_k^0(z_0)}{} = \frac{\dd z_0}{2}\bigg\{ \frac{1}{2}\,\dd_{z_0}\left(\Res_{z = \infty} \frac{x(z)^{k}}{k!} \frac{\ln z_0\,\dd x(z)}{x(z) - x(z_0)}\right)\\
\hphantom{\eta_k^0(z_0) =}{}
 + \frac{2}{k!} \frac{\partial}{\partial \epsilon} \big[z^{k - 1}\big]\left(\sum_{c \geq 0} (c + 1)z_0^{c}z^{c}\big(1 + z^2\big)^{k + \epsilon}\right)\bigg|_{\epsilon = 0}\bigg\} \\
\hphantom{\eta_k^0(z_0) =}{} -\frac{2H_k}{k!}\big[z^{k - 1}\big]\bigg(\sum_{c \geq 0} (c + 1)z_0^{c}z^{c} \big(1 + z^2\big)^{k}\bigg) - (z_0 \rightarrow 1/z_0) \\
\hphantom{\eta_k^0(z_0)}{} = \dd_{z_0}\left(\frac{x(z_0)^{k}}{k!}\,\tfrac{1}{2}\big(\ln z_0 - \ln\big(z_0^{-1}\big)\big) - \sum_{\substack{b,c \geq 0 \\ 2b + c = k - 1}} \frac{H_{k - b}}{b!(k - b)!}\big(z_0^{c + 1} - z_0^{-(c + 1)}\big)\right).\tag*{\qed}
\end{gather*}\renewcommand{\qed}{}
\end{proof}

\section{Proof of Theorem~\ref{insertions}} \label{proofthm1}

\subsection[The $S$-matrix]{The $\boldsymbol{S}$-matrix}

This subsection relies on Proposition~\ref{PROPJIU} and does not need the computations with the operators beyond Section~\ref{S53}.
\begin{proposition} \label{Sixi}
For any $a,b \geq 0$ and $\alpha,\beta \in \{0,1\}$, we have $\ci^\alpha_a\big[\xi^\beta_b\big] =(S_{a-b})_\alpha^\beta$.
\end{proposition}
Together with Proposition~\ref{PROPJIU}, we find the table of non-zero entries of $S_k$ is
\begin{gather*}
\renewcommand{\arraystretch}{1.6}
\begin{array}{|c|c|c|}
\hline
(S_{k})_{\alpha}^{\beta}& \alpha = 0 & \alpha = 1 \bsep{4pt}\\
\hline
\beta = 0 & \dfrac{1 - 2mH_{2m}}{m!^2}\,\,\,({\rm even}) & \dfrac{1}{m!(m + 1)!}\,\,\,({\rm odd})\tsep{8pt}\bsep{8pt}\\
\hline
\beta = 1 & \dfrac{-2H_{m}}{m!^2}\,\,\,({\rm odd}) & \dfrac{1}{m!^2}\,\,\,({\rm even}) \tsep{8pt}\bsep{8pt}\\
\hline
\end{array}
\renewcommand{\arraystretch}{1.0}
\end{gather*}
where ``even'', respectively ``odd'', means that $k = 2m$, respectively $k = 2m + 1$.

\begin{proof}
By integration by parts \eqref{IPP222} and Lemma~\ref{IPP}, it suffices to prove the result for $\ell = 0$. Kontsevich and Manin \cite{KMaRel} show that for $k>0$
\begin{gather} \label{SKM}
(S_k)_\alpha^\beta=\big\langle\tau_0^{1-\beta}\tau_{k-1}^{\alpha}\big\rangle_{0}.
\end{gather}
The divisor equation and genus 0 topological recursion relations satisfied quite generally by Gromov--Witten invariants \cite{KMaGro} can be used to calculate the right-hand side of \eqref{SKM}. We have
\begin{gather*}
\big\langle\tau_0^1\tau_0^{1-\beta}\tau_{k-1}^\alpha\big\rangle= \big\langle\tau_{k-2}^\alpha\tau_0^\beta\big\rangle\big\langle\tau_0^{1-\beta}\tau_0^1\tau_0^{1-\beta}\big\rangle=\big\langle\tau_{k-2}^\alpha\tau_0^\beta\big\rangle,
\end{gather*}
where the first equality is the genus 0 topological recursion relation and $\big\langle\tau_0^{1-\beta}\tau_0^1\tau_0^{1-\beta}\big\rangle=1$ for $\beta=0$ or 1 gives the second equality. The divisor equation allows one to remove an insertion $\tau_0^1$ and in this case gives
\begin{gather*}
\big\langle\tau_0^1\tau_0^{1-\beta}\tau_{k-1}^\alpha\big\rangle=d\big\langle\tau_0^{1-\beta}\tau_{k-1}^\alpha\big\rangle+\delta_{\alpha,0} \big\langle\tau_0^{1-\beta}\tau_{k-2}^{1-\alpha}\big\rangle,
\end{gather*}
where $d=\frac12(k+\alpha-\beta)$ is the degree. Putting these together and using~\eqref{SKM}, one gets for $k>0$
\begin{gather*}
\tfrac12(k+\alpha-\beta)\,(S_k)^{\beta}_\alpha= (S_{k-1} )^{1-\beta}_\alpha-\delta_{\alpha,0} (S_{k-1})^{\beta}_{1-\alpha},
\end{gather*}
which uniquely determines $S_k$ from $(S_0)_{\alpha}^{\beta}= \delta_{\alpha,\beta}$ and $(S_1)_0^1= 0$. Then, one can check that the values of $\mathcal{I}_{k}^{\alpha}[\xi_{0}^{\beta}]$ given in Proposition~\ref{PROPJIU} satisfy the same recursion with the same initial conditions.
\end{proof}

We can express the $S$-matrix with respect to the odd part of the $(0,2)$-correlator
\begin{gather*}
\omega^{\rm odd}_{0,2}(z_1,z_2)=\frac{1}{2}\big(\omega_{0,2}(z_1,z_2)-\omega_{0,2}(z_1,1/z_2)\big).
\end{gather*}
This quantity is invariant under exchange of $z_1$ and $z_2$ since $\omega_{0,2}(z_1,1/z_2)=\omega_{0,2}(1/z_1,z_2)$.
\begin{corollary} \label{th:Sw2} For any $b > 0$, and $\alpha,\beta \in \{0,1\}$, we have
\begin{gather*}
(S_b)^{\alpha}_\beta=\ci^{\beta}_{b-1}\otimes\ci^{1-\alpha}_0\big[\omega_{0,2}^{\rm odd}\big].
\end{gather*}
\end{corollary}
\begin{proof}
For $\alpha = 0$, we compute
\begin{gather*}
\mathcal{I}_{0}^{1}\big[\omega_{0,2}^{{\rm odd}}(\cdot,z_1)\big] = - \Res_{z = \infty} \frac{(z + 1/z)\dd z\,\dd z_1}{2}\left(\frac{1}{(z - z_1)^2} + \frac{1}{(1 - z_1z)^2}\right)\\
\hphantom{\mathcal{I}_{0}^{1}\big[\omega_{0,2}^{{\rm odd}}(\cdot,z_1)\big]}{} = \frac{\big(1 + 1/z_1^{2}\big)\dd z_1}{2} = \xi_{-1}^{0,{\rm odd}}(z_1),
\end{gather*}
by comparison with \eqref{xiodd}. For $\alpha = 1$, we compute using Lemma~\ref{IntL}
\begin{gather*}
\mathcal{I}_{0}^{0}\big[\omega_{0,2}^{{\rm odd}}(\cdot,z_1)\big] = -\frac{\dd z_1}{2} \int_{0}^{\infty} \left(\frac{1}{(z - z_1)^2} + \frac{1}{(zz_1 - 1)^2}\right) = \frac{\dd z_1}{z_1} = \xi_{-1}^{1,{\rm odd}}(z_1).
\end{gather*}
Therefore, for any $\alpha \in \{0,1\}$, the second evaluation of $\mathcal{I}_{b - 1}^{\beta}$ is well-defined and we have
\begin{gather*}
\mathcal{I}_{b - 1}^{\beta} \otimes \mathcal{I}_{0}^{1 - \alpha}\big[\omega_{0,2}^{{\rm odd}}\big] = \mathcal{I}_{b - 1}^{\beta}\big[\xi_{-1}^{\alpha,{\rm odd}}\big].
\end{gather*}
Using Remark~\ref{oddre} and the properties of integration by parts \eqref{IPP222} and Lemma~\ref{IPP}, we get
\begin{gather*}
\mathcal{I}_{b - 1}^{\beta} \otimes \mathcal{I}_{0}^{1 - \alpha}\big[\omega_{0,2}^{{\rm odd}}\big] = \mathcal{I}_{b}^{\beta}\big[\xi_{0}^{\alpha}\big],
\end{gather*}
which is equal to $(S_{b})_{\beta}^{\alpha}$ according to Proposition~\ref{Sixi}.
\end{proof}

\subsection{The stable cases}

The cases $2g - 2 + n > 0$ of Theorem~\ref{insertions} can be deduced purely by linear algebra starting from Theorem~\ref{th:dboss} -- due to \cite{DBOSS} -- and Proposition~\ref{Sixi}. Indeed,
\begin{align*}
\bigotimes_{i = 1}^n \mathcal{I}_{a_i}^{\alpha_i} [\omega_{g,n}] & = \bigotimes_{i = 1}^n \mathcal{I}_{b_i}^{\alpha_i}\bigg[\sum_{\substack{k_1,\ldots,k_n \geq 0 \\ \beta_1,\ldots,\beta_n \in \{0,1\}}} \bigg\langle\prod_{i=1}^n\overline{\tau}_{k_i}^{\beta_i}\bigg\rangle_g\,\bigotimes_{i=1}^n\xi_{k_i}^{\beta_i}\bigg] \\
& = \sum_{\substack{k_1,\ldots,k_n \geq 0 \\ \beta_1,\ldots,\beta_n \in \{0,1\}}} \bigg\langle\prod_{i=1}^n\overline{\tau}_{k_i}^{\beta_i}\bigg\rangle_g\,\prod_{i=1}^n\mathcal{I}_{a_i}^{\alpha_i}\big[\xi_{k_i}^{\beta_i}\big] \\
&=\sum_{\substack{k_1,\ldots,k_n \geq 0 \\ \beta_1,\ldots,\beta_n \in \{0,1\}}} \bigg\langle\prod_{i=1}^n\overline{\tau}_{k_i}^{\beta_i}\bigg\rangle_g \prod_{i=1}^n (S_{a_i-k_i})_{\alpha_i}^{\beta_i} \\
&= \bigg\langle\prod_{i=1}^n\tau_{b_i}^{\alpha_i}\bigg\rangle_g,
\end{align*}
where the first equality uses Theorem~\ref{th:dboss} and the final equality uses \eqref{Srel}.

\subsection[The $(0,1)$ cases]{The $\boldsymbol{(0,1)}$ cases}\label{S63}
We can prove the $(0,1)$ case of Theorem~\ref{insertions} in the following way. We use the string equation to represent $\big\langle \tau_b^{\beta}\big\rangle_{0} = \big\langle \tau_{0}^{0} \tau_{b + 1}^{\beta}\big\rangle_{0}$. This is equal to $(S_{b + 2})_{\beta}^1$ according to \eqref{SKM}, hence equal to $\mathcal{I}_{b + 2}^{\beta}\big[\xi_{0}^1\big]$ owing to Corollary~\ref{th:Sw2}. Using the integration by parts~\eqref{IPP222} and Lemma~\ref{IPP} it is also $\mathcal{I}_{b + 1}^{\beta}\big[\xi_{-1}^{1,{\rm odd}}\big] = \mathcal{I}_{b + 1}^{\beta}[\dd z/z]$, which is the formula we sought for. According to Proposition~\ref{PROPJIU}, the table of values is
\begin{gather*}
\renewcommand{\arraystretch}{1.6}
\begin{array}{|c|c|c|}
\hline \big\langle \tau_{b}^{\beta}\big\rangle_0 & \beta = 0 & \beta = 1\bsep{4pt} \\
\hline
b = 2m & 0 & \dfrac{1}{(m + 1)!^2} \tsep{8pt}\bsep{8pt}\\
\hline
b = 2m + 1 & -\dfrac{2H_{m + 1}}{(m + 1)!^2} & 0 \tsep{8pt}\bsep{8pt}\\
\hline
\end{array}
\renewcommand{\arraystretch}{1.0}
\end{gather*}

\subsection[The $(0,2)$ case]{The $\boldsymbol{(0,2)}$ case}\label{S64}
Likewise, Corollary~\ref{th:Sw2} gives the $(0,2)$ case of Theorem~\ref{insertions} with one insertion of degree~0 since via~\eqref{SKM} it can be restated as
\begin{gather} \label{(0,2)}
\big\langle\tau_{b}^{\beta}\tau_0^{\alpha}\big\rangle_0 =\ci^{\beta}_{b}\otimes\ci^{\alpha}_0\big[\omega^{\rm odd}_{0,2}\big].
\end{gather}
The proof of Theorem~\ref{insertions} is completed through the following proposition.
\begin{proposition}
\begin{gather*}
\big\langle\tau^\alpha_{j}\tau^\beta_k\big\rangle_0=\ci^\alpha_{j}\otimes\ci^\beta_k\big[\omega_{0,2}^{\rm odd}\big].
\end{gather*}
\end{proposition}
\begin{proof}By direct evaluation, we have for $a \in \{-1,1\}$
\begin{gather*}
\Res_{z= a}\frac{\omega_{0,2}(z_1,z)\omega_{0,2}(z,z_ 2)}{\dd x(z)}= \frac{a}{2}\,\frac{\dd z_1}{(z_1- a)^2}\,\frac{\dd z_2}{(z_2 - a)^2},
\end{gather*}
and we sum these to get
\begin{gather*}
\sum_{a = \pm 1} \Res_{z=a}\frac{\omega_{0,2}(z_1,z)\omega_{0,2}(z,z_ 2)}{\dd x(z)}=\xi_0^0(z_1)\xi_0^1(z_2)+\xi_0^1(z_1)\xi_0^0(z_2),
\end{gather*}
where $\xi_0^\beta(z)$, defined in \eqref{auxdif}, have order two poles at $z=\pm1$. Now
\begin{align*}
\sum_{a = \pm 1} \Res_{z=a}\frac{\omega_{0,2}(z_1,z)\omega_{0,2}(z,z_2)}{\dd x(z)}&=
-\Res_{z=z_1}\frac{\omega_{0,2}(z_1,z)\omega_{0,2}(z,z_2)}{\dd x(z)}-\Res_{z=z_2}\frac{\omega_{0,2}(z_1,z)\omega_{0,2}(z,z_2)}{\dd x(z)}\\
& =-\dd_1\left(\frac{\omega_{0,2}(z_1,z_2)}{\dd x(z_1)}\right)-\dd_{2}\left(\frac{\omega_{0,2}(z_1,z_2)}{\dd x(z_2)}\right),
\end{align*}
where the first equality uses the fact that the only poles of the integrand are $1$, $-1$, $z_1$, $z_2$. Putting these together gives
\begin{gather}
\label{unberg0}
-\dd_1\left(\frac{\omega_{0,2}(z_1,z_2)}{\dd x(z_1)}\right)-\dd_{2}\left(\frac{\omega_{0,2}(z_1,z_2)}{\dd x(z_2)}\right)=\xi_0^0(z_1)\xi_0^1(z_2)+\xi_0^1(z_1)\xi_0^0(z_2).
\end{gather}
Take $\frac12\eqref{unberg0}-\frac12\eqref{unberg0}|_{z_2\mapsto 1/z_2}$ and since $\xi^\alpha_0(z_2)$ is odd under this involution, we get
\begin{gather} \label{unberg}
-\dd_1\left(\frac{\omega^{\rm odd}_{0,2}(z_1,z_2)}{\dd x(z_1)}\right)-\dd_{2}\left(\frac{\omega_{0,2}^{\rm odd}(z_1,z_2)}{\dd x(z_2)}\right)=\xi_0^0(z_1)\xi_0^1(z_2)+\xi_0^1(z_1)\xi_0^0(z_2).
\end{gather}
We apply $\ci^\alpha_j\otimes\ci^\beta_k$ to \eqref{unberg} to get
\begin{gather*}
\ci^\alpha_j\otimes\ci^\beta_k\left[
-\dd_1\left(\frac{\omega_{0,2}^{\rm odd}(z_1,z_2)}{\dd x(z_1)}\right)-\dd_{2}\left(\frac{\omega_{0,2}^{\rm odd}(z_1,z_2)}{\dd x(z_2)}\right)\right]\\
\qquad{} =\ci^\alpha_j\otimes\ci^\beta_k\big[\xi_0^0(z_1)\xi_0^1(z_2)+\xi_0^1(z_1)\xi_0^0(z_2)\big],
\end{gather*}
hence
\begin{gather}
 \label{2ptprim}
\ci^\alpha_{j-1}\otimes\ci^\beta_k\big[\omega_{0,2}^{\rm odd}\big]+\ci^\alpha_{j}\otimes\ci^\beta_{k-1}\big[\omega_{0,2}^{\rm odd}\big]=\big\langle\tau_0^1\tau_{j-1}^\alpha\big\rangle_0\big\langle\tau_0^0\tau_{k-1}^\beta\big\rangle_0+ \big\langle\tau_0^0\tau_{j-1}^\alpha\big\rangle_0\big\langle\tau_0^1\tau_{k-1}^\beta\big\rangle_0,
\end{gather}
where on the left-hand side we have used integration by parts, and on the right-hand side we have used $\ci^\alpha_j\big[\xi^\beta_0(z)\big]=\big\langle\tau^{1-\beta}_0\tau_{j-1}^\alpha\big\rangle_{0}$ from Proposition~\ref{Sixi} and~\eqref{SKM}. Notice that \eqref{2ptprim} determines $\ci^\alpha_{j}\otimes\ci^\beta_k\big[\omega_{0,2}^{\rm odd}\big]$ inductively from the case $j=0$ and the two-points genus zero descendant Gromov--Witten invariants with one primary insertion (degree $0$) which appear on the right-hand side of~\eqref{2ptprim}. The $j=0$ case has already been shown in \eqref{(0,2)} to give genus zero descendant Gromov--Witten invariants.

The genus zero descendant Gromov--Witten invariants satisfy
\begin{gather*}
\big\langle\tau^\alpha_{j-1}\tau^\beta_k\big\rangle_0+\big\langle\tau^\alpha_{j}\tau^\beta_{k-1}\big\rangle_0= \big\langle\tau^0_0\tau^\alpha_{j}\tau^\beta_k\big\rangle_0= \big\langle\tau_0^1\tau_{j-1}^\alpha\big\rangle_0\big\langle\tau_0^0\tau_{k-1}^\beta\big\rangle_0+ \big\langle\tau_0^0\tau_{j-1}^\alpha\big\rangle_0\big\langle\tau_0^1\tau_{k-1}^\beta\big\rangle_0,
\end{gather*}
where the first equality is the string equation and the second equality is the genus zero topological recursion relation together with the string equation. Hence the two-points genus zero descendant Gromov--Witten invariants are also determined inductively from the two-points genus zero descendant Gromov--Witten invariants with one primary insertion, via the relation analogous to~\eqref{2ptprim}. So we conclude
\begin{gather*}
\big\langle\tau^\alpha_{j}\tau^\beta_k\big\rangle_0=\ci^\alpha_{j}\otimes\ci^\beta_k\big[\omega_{0,2}^{\rm odd}\big]
\end{gather*}
as required.
\end{proof}

\section[New proof of global Virasoro constraints for $\bp^1$]{New proof of global Virasoro constraints for $\boldsymbol{\bp^1}$}\label{S7}

The Virasoro constraints in Theorem~\ref{th:vir} allow the removal of non-stationary insertions so that the stationary invariants determine the non-stationary insertions.

\subsection{Decay rules}\label{Sdecay}
Okounkov and Pandharipande view the global Virasoro constraints as the decay of non-statio\-na\-ry insertions which are considered to be unstable. The decay rules for~$L_{k}$ defined in \eqref{viror} are as follows:
\begin{gather}
\tau_{k+1}^0\tau_\ell^0 \mapsto\binom{k+\ell}{\ell-1}\tau_{k+\ell}^0, \label{D1}\\
\tau_{k+1}^0\tau_\ell^0 \mapsto2\binom{k+\ell}{\ell-1}(H_{k+\ell}-H_{\ell-1})\tau_{k+\ell-1}^1, \label{D2}\\
\tau_{k+1}^0\tau_\ell^1\mapsto\binom{k+\ell+1}{\ell}\tau_{k+\ell}^1, \label{D3}\\
\tau_{k+1}^0 \mapsto-2H_{k+1}\tau_{k}^1, \label{D4}\\
\tau_{k+1}^0 \mapsto\frac{\hbar}{k+1}\sum_{m=0}^{k-2}\binom{k}{m+1}^{-1}\tau_m^1\tau_{k-m-2}^1. \label{D5}
\end{gather}
This means that we sum over interactions of $\tau_{k+1}^0$ with all other insertions.
For example, when $k=-1$, \eqref{D1} and \eqref{D3} become
$\tau_0^0\tau_\ell^\alpha\mapsto\tau_{\ell-1}^\alpha$ and \eqref{D2}, \eqref{D4} and \eqref{D5} produce zero -- where we use the convention that a sum vanishes if its upper terminal is negative. Summing over these contributions, we see that $L_{-1}$ produces the string equation:
\begin{gather*}%\label{stringsss}
\bigg\langle\tau_0^0\prod_{i=1}^n\tau_{a_i}^{\alpha_i}\bigg\rangle_g = \sum_{j=1}^n\big\langle\tau_{a_1}^{\alpha_1}\cdots\tau_{a_j-1}^{\alpha_j}\cdots\tau_{a_n}^{\alpha_n}\big\rangle_g.
\end{gather*}

The application of the operators to $\omega_{0,1}$ and $\omega_{0,2}$ was computed in Sections~\ref{S63}--\ref{S64} and their values can be checked to satisfy these decay rules. In fact, the local constraints \eqref{qgns} for $(g,n) = (0,1)$ and $(0,2)$ have an empty content -- one rather considers the local constraints for \smash{$2g - 2 + n > 0$} for given $(\omega_{0,1},\omega_{0,2})$, which determine uniquely the $\omega_{g,n}$ for $2g - 2 + n > 0$. We are now going to prove from the local constraints~\eqref{qgns} for $2g - 2 + n > 0$ that the global Virasoro constraints are satisfied.

\subsection{Proof of Theorem~\ref{th:vir}}

The cases $(0,3)$ and $(1,1)$ are treated separately below, so we assume $2g - 2 + n \geq 2$. The proof is achieved by applying $\bigotimes_{i=1}^n\ci^{\alpha_i}_{b_i}$ to both sides of the identity in Lemma~\ref{Vircomp}. It can be rewritten
\begin{gather}
\mathcal{A}\big[\omega_{g,n}(\cdot,\zz_I)\big](x_1)\,\dd x_1 + 2\ln(x(z_1)/z_1)\,\omega_{g,n}(z_1,\zz_I) \nonumber\\
\qquad{} = \sum_{i = 2}^n \dd_i\left(\frac{\dd x_1}{x_1 - x_i}\,\frac{\omega_{g,n - 1}(\zz_{I})}{\dd x_i}\right)
 + \frac{\omega_{g - 1,n + 1}(z_1,z_1,\zz_I) }{\dd x_1} \nonumber\\
 \qquad\quad{} + \sum_{\substack{h + h' = g \\ J \sqcup J' = I}}^{\circ} \frac{\omega_{h,1 + |J|}(z_1,\zz_J)\omega_{h',1 + |J'|}(z_1,\zz_{J'})}{\dd x_1} + \frac{2\tilde{\omega}_{0,1}(z_1)\omega_{g,n}(z_1,\zz_I)}{\dd x_1} \nonumber\\
\qquad{} = \sum_{i = 2}^n \mathcal{L}[\omega_{g,n - 1}(\cdot,\zz_I)](z_1)
 + \frac{\omega_{g - 1,n + 1}(z_1,z_1,\zz_I)}{\dd x_1} \nonumber\\
 \qquad\quad{} + \sum_{\substack{h + h' = g \\ J \sqcup J' = I}}^{\circ\circ} \frac{\omega_{h,1 + |J|}(z_1,\zz_J)\omega_{h',1 + |J'|}(z_1,\zz_{J'})}{\dd x_1} + \frac{2\tilde{\omega}_{0,1}(z_1)\omega_{g,n}(z_1,\zz_I)}{\dd x_1}.\label{diffdecay}
\end{gather}
We have added artificially the term containing $\tilde{\omega}_{0,1} = \ln(x/z)\dd x$ on both sides so as to exploit Corollary~\ref{LEPTH}.

Fix $\alpha_1=1$, $b_1=k$ and allow $\alpha_i$ and $b_i$ to be arbitrary for $i>1$. We first apply $\bigotimes_{i = 2}^n \mathcal{I}_{b_i}^{\alpha_i}$, and then $\mathcal{I}_{k}^1$ to the first variable $z_1$. We claim that the evaluation of $\bigotimes_{i=1}^n\ci^{\alpha_i}_{b_i}$ on the right-hand side of \eqref{diffdecay} is in fact independent of the order in which we evaluate each $\ci^{\alpha_i}_{b_i}$. It is not obvious for the terms involving $\mathcal{L}$, but the decomposition proved in Lemma~\ref{Ltransform} shows that the operators $\ci^1_{b_1}$, respectively $\ci^{\alpha_i}_{b_i}$, naturally act on the variable $z_1$, respectively the variable $z_i$, in $\cl\big(\xi_m^\alpha\big)(z_1,z_i)$ and they commute.

Applying $\otimes_{i = 2}^n \mathcal{I}_{b_i}^{\alpha_i}$ on the left-hand side of \eqref{diffdecay} gives a $1$-form in the variable $x_1$, to which we can apply $\mathcal{I}_{k}^1$ using Corollary~\ref{LEPTH}{\samepage
\begin{gather}
 \bigotimes_{i = 1}^n \mathcal{I}_{b_i}^{\alpha_i} \big(\mathcal{A}\big[\omega_{g,n}(\cdot,\zz_I)\big](x_1)\,\dd x_1+2\omega_{g,n}(z_1,\zz_{I})\ln(x_1/z_1)\big) \nonumber\\
 \qquad{} = \bigg\langle \tau_{k + 1}^0 \prod_{i=2}^n\tau^{\alpha_i}_{b_i}\bigg\rangle_g + 2H_{k+1 }
 \bigg\langle\tau_{k}^1\prod_{i=2}^n\tau^{\alpha_i}_{b_i}\bigg\rangle_g,\label{evalog}
\end{gather}
which reproduces the decay rule \eqref{D4}.}

The decay rule \eqref{D5} means that one inserts $\tau_m^1\tau_{k-m-2}^1$ into a $(g-1,n+2)$ correlator or into the product of $(g_1,n_1+1)$ and $(g_2,n_2+1)$ correlators for $g_1+g_2=g$ and $n_1+n_2=n$. This is reproduced by applying first $\mathcal{I}_{k}^1$, then $\otimes_{i = 2}^n \mathcal{I}_{b_i}^{\alpha_i}$ to the second line of \eqref{diffdecay} as follows. Note that $\tilde{\omega}_{0,1}(z)= \ln(x/z)\,\dd x$ is analytic at $z=\infty$, and we use $\mathcal{I}_{k}^1[\tilde{\omega}_{0,1}](z) =\langle\tau^1_k\rangle_0$ to get the $(0,1)$ contribution $ \langle\tau^1_k\rangle_0$ arising from \eqref{D5}.
The action of $\ci^1_k$ uses only the expansion at $x_1=\infty$. Hence from the second line of \eqref{diffdecay} we can write the insertions as follows:
\begin{gather*}
\sum_{i \geq 0} \tau_i^1 \frac{(i+1)!}{x_1^{i + 2}}\cdot\sum_{j \geq 0} \tau_j^1 \frac{(j+1)!}{x_1^{j+ 2}}=
\sum_{k \geq 2} \sum_{m=0}^{k-2} \tau_m^1\tau_{k-m-2}^1\frac{(m+1)!(k-m-1)!}{x_1^{k+2}},
\end{gather*}
thus we have
\begin{gather*}
\bigg\langle\tau^0_{k+1}\prod_{i=2}^n\tau^{\alpha_i}_{b_i}\bigg\rangle_g
 =-2H_{k+1}\bigg\langle\tau^1_{k}\prod_{i=2}^n\tau^{\alpha_i}_{b_i}\bigg\rangle_g
+ \bigotimes_{i=1}^n\ci^{\alpha_i}_{b_i}\cdot\sum_{j=2}^n \dd_j\left\{\frac{\omega_{g,n - 1}(\zz_{I})}{(x_1 - x_j)\,\dd x_j}\right\} \\
\qquad{} +\frac{1}{k+1}\sum_{m=0}^{k-2}\binom{k}{m+1}^{-1}\Bigg[\big\langle\tau_{m}^1\tau_{k-m-2}^1\tau_{\mathbf{b}_I}^{\boldsymbol{\alpha}_I} \big\rangle_{g-1}+\sum_{\substack{h+h'=g \\ J \sqcup J'=I}} \big\langle\tau_m^1\tau_{\mathbf{b}_I}^{\boldsymbol{\alpha}_I}\big\rangle_h \big\langle\tau_{k-m-2}^1\tau_{\mathbf{b}_J}^{\boldsymbol{\alpha}_J}\big\rangle_{h'}\Bigg].
\end{gather*}

We will now deal with the remaining unevaluated term which is the first term on the right-hand side of \eqref{diffdecay}. For the $i$th summand, we first apply $\otimes_{j \neq 1,i} \mathcal{I}_{b_j}^{\alpha_j}$, then $\mathcal{I}_{b_1}^{\alpha_1} = \mathcal{I}_k^1$, and finally $\mathcal{I}_{b_i}^{\alpha_i}$. In the process we use the fact that $\omega_{g,n - 1}(\zz_I)$ is a linear combination of $\xi_m^\alpha(z_i)$ with coefficients given by differentials in $(z_j)_{j \neq 1,i}$, and the following computation, which we are going to use for $\beta = \alpha_i$ and $j = b_i$
\begin{gather}
 \ci^\beta_j\ci^1_k\left\{-\dd_i\left(\frac{\dd x_1}{x_1 - x_i} \frac{\xi_m^{\alpha}(z_i)}{\dd x_i}\right)\right\}
=-\ci^\beta_j\,\dd_i\left(\frac{x_i^{k+1}}{(k+1)!}\frac{\xi_m^\alpha(z_i)}{\dd x_i}\right) = \ci^\beta_{j-1}\left(\frac{x_i^{k+1}}{(k+1)!}\xi_m^\alpha(z_i)\right) \nonumber\\
\qquad{}= \binom{j+k+\beta}{k+1} \ci^\beta_{j+k}\big(\xi_m^\alpha\big) + 2\delta_{0,\beta}\binom{j+k}{k+1}(H_{j + k} - H_{j - 1})\ci^1_{j+k-1}\big(\xi_m^\alpha\big) \nonumber\\
\qquad{}=\binom{j+k+\beta}{k+1} (S_{j+k-m})_\beta^\alpha + 2\delta_{0,\beta}\binom{j+k}{k+1}(H_{j + k} - H_{j - 1})(S_{j+k-m-1})_1^\alpha \nonumber\\
\qquad{}=\binom{j+k+\beta}{k+1} \big\{(S_{j+k-m})_\beta^\alpha + 2\delta_{0,\beta}(H_{j+k}-H_{j-1}) (S_{j+k-m-1})_1^\alpha\big\}.\label{bidifferentialdecay}
\end{gather}

The first equality transforms a differential with poles at $z_2=1,-1,z_1,1/z_1$ to a differential with poles only at $z_2=1,-1$. This is achieved by evaluating $\ci^1_k$ first, which crucially depends on Lemma~\ref{Ltransform} which guarantees that the operators $\ci^{\alpha_i}_{b_i}$ and $\ci^1_k$ commute when applied to \eqref{diffdecay}. Note that they do not commute when applied to individual terms like $\dd_i\big(\frac{\xi_m^{\alpha}(z_i)}{(x_1 - x_i)\dd x_i}\big)$ and terms involving $\omega_{0,2}$. We can nevertheless use linearity and apply $\ci^1_k$ first to those special terms in~\eqref{diffdecay}. The third equality in~\eqref{bidifferentialdecay} uses Proposition~\ref{xaction}.

We see in \eqref{bidifferentialdecay} that a single term appears on the right-hand side when $\beta=1$ which we will see corresponds to the decay rule \eqref{D3}, whereas two terms appear on the right-hand side when $\beta=0$, which we will see corresponds to the decay rules \eqref{D1} and \eqref{D2}. Indeed, by \cite{DBOSS}, the coefficient of~$\xi_m^\alpha(z_i)$ (in $\omega_{g,n - 1}(\zz_I)$ in the case of concern here) gives the insertion of $\bar{\tau}_m^\alpha$. Hence~\eqref{bidifferentialdecay} proves the decay rules~\eqref{D1}, \eqref{D2} and \eqref{D3} which are given in terms of ancestor invariants via~\eqref{Srel}
\begin{align*}
\tau_{k+1}^0\tau_j^0&\mapsto\binom{j+k}{k+1}\big(\tau_{j+k}^0+2(H_{j+k}-H_{j-1})\tau_{j+k-1}^1\big)\\
&=\sum_{m,\alpha}\binom{j+k}{k+1}\big((S_{j+k-m})^{\alpha}_0+2(H_{j+k}-H_{j-1})(S_{j+k-1-m})^{\alpha}_1\big)\overline{\tau}_m^{\alpha}.
\end{align*}

 \textbf{$\boldsymbol{(1,1)}$ case.}
For the $(1,1)$ case, we apply $\ci^1_{k}$ to the identity proven in Lemma~\ref{Vircomp}:
\begin{gather*}
\mathcal{A}[\omega_{1,1}](x_1)\dd x_1 + \omega_{0,2}(z_1,1/z_1) = 0.
\end{gather*}
This case essentially uses the argument above, with a minor variation to deal with the term $ \omega_{0,2}(z_1,1/z_1)$. We use~\eqref{evalog} to produce the decay rule~\eqref{D4}. To get the decay rule~\eqref{D5}, we have
\begin{align*}
\Omega_{0,2}(x_1,x_2) & =\sum_{b_1,b_2 \geq 0}\big\langle \tau_{b_1}^1\tau_{b_2}^1 \big\rangle_0 \frac{(b_1+1)!}{x_1^{b_1 + 2}}\frac{(b_2+1)!}{x_2^{b_2 + 2}}\dd x_1 \dd x_2 \\
& \sim \omega_{0,2}(z_1,z_2)-\frac{\dd x_1\dd x_2}{(x_1-x_2)^2} =-\omega_{0,2}(z_1,1/z_2).
\end{align*}
Hence $\omega_{0,2}(z_1,1/z_1)\sim-\Omega_{0,2}(x_1,x_1)$ and the action of $\ci^1_k$ uses only the expansion at $x_1=\infty$, so as described above this yields the decay rule \eqref{D5}.

\textbf{$\boldsymbol{(0,3)}$ case.} According to Lemma~\ref{Vircomp}, we have
\begin{gather*}
\mathcal{A}[\omega_{0,3}(\cdot,z_2,z_3)](x_1)\dd x_1 = \dd_2\left(\frac{\dd x_1}{x_1 - x_2}\,\frac{\omega_{0,2}^{\rm odd}(z_2,z_3)}{\dd x_2}\right) + \dd_3\left(\frac{\dd x_1}{x_1 - x_3}\,\frac{\omega_{0,2}^{\rm odd}(z_2,z_3)}{\dd x_3}\right)\\
\hphantom{\mathcal{A}[\omega_{0,3}(\cdot,z_2,z_3)](x_1)\dd x_1 =}{} - \frac{2\omega_{0,2}^{\rm odd}(z_1,z_2)\omega_{0,2}^{\rm odd}(z_1,z_3)}{\dd x_1}.
\end{gather*}
Using similar methods to those above, we can identify each decay rule by applying $\ci^1_k\otimes\ci^{\alpha_2}_{b_2}\otimes\ci^{\alpha_3}_{b_3}$ to this relation. We omit the details, since the computation of the $(0,3)$ descendant invariants follow easily from the genus 0 topological recursion relations~\cite{KMaGro}.

\appendix

\section{Evaluation on invariant differentials}

In the appendix, we evaluate the action of $\mathcal{I}_{k}^0$ on differentials invariant under the involution $z\mapsto 1/z$. Although this is not used in the article, we include it to give a fuller understanding of the operator $\ci^0_k$.

The operators $\ci^0_k$ are regularised integrals along a contour between the two points above $x=\infty$, given by $z=0$ and $z=\infty$. Consider such an integral applied to a meromorphic $1$-form~$f$ on $\mathcal{S}$ invariant under the involution, \textit{i.e.} satisfying $f(z) = f(1/z)$, hence obtained as a pullback of a meromorphic $1$-form from the $x$-plane $\check{\mathcal{S}}$. It would be reasonable to expect that such an integral would vanish since the contour downstairs seems to be closed. Here we will see that $\ci^0_k$ applied to the pullback of a meromorphic $1$-form from the $x$-plane can in fact be non-zero.

The pullback of a meromorphic $1$-form from the $x$-plane can always be obtained by integrating
\begin{gather*}
\check{B}(x,x_0) = \frac{\dd x\,\dd x_0}{(x - x_0)^2}
\end{gather*}
with respect to $x_0$ on a suitable current in $\check{\mathcal{S}}$. Therefore, the evaluation of the operators on meromorphic $1$-forms is determined by their evaluation on $\check{B}(\cdot,x_0)$. The evaluation of $\mathcal{I}^1$ is straightforward. We describe here the evaluation of $\mathcal{I}^0$.

\begin{proposition}
Let $x_0 \in \mathbb{C} \setminus {\rm i}\mathbb{R}$. We have
\begin{gather*}
\mathcal{I}_{0}^0\bigg[\frac{\check{B}(\cdot,x_0)}{\dd x_0}\bigg] = \frac{2}{x_0},
\end{gather*}
and for $b > 0$
\begin{gather*}
\mathcal{I}_{b}^0\bigg[\frac{\check{B}(\cdot,x_0)}{\dd x_0}\bigg] = \frac{x_0^{b - 1}}{(b - 1)!}\,\ln\big(x_0^2\big) + \frac{2x_0^{b - 1}}{b!}\,(1 - bH_{b}).
\end{gather*}
\end{proposition}

\begin{proof}
We let $f_{x_0} = \check{B}(\cdot,x_0)/\dd x_0$. The case $b = 0$ does not need regularisation and is easy to compute. For $b > 0$, we have
\begin{gather*}
\mathcal{I}_{b - 1}^1[f_{x_0}] = \frac{x_0^{b - 1}}{(b - 1)!},\qquad \check{\mathcal{R}}_{b}[f_{x_0}](x) = \frac{1}{(x - x_0)^2} - \sum_{a = 0}^{b - 1} \frac{(a + 1)x_0^{a}}{x^{a + 2}}.
\end{gather*}
Therefore
\begin{gather}
\frac{(x - {\rm i}\epsilon)^{b}}{b!}\,\frac{\check{\mathcal{R}}_{b}[f_{x_0}](x)}{\dd x}
 = \frac{(x - {\rm i}\epsilon)^b}{(x - x_0)^2} - \sum_{a = 0}^{b - 1} \sum_{c = 0}^{b} \frac{(a + 1) x^{c - (a + 2)} (-{\rm i}\epsilon)^{b - c}x_0^{a}}{c!(b - c)!} \nonumber\\
\quad{} = \frac{(x_0 - {\rm i}\epsilon)^{b - 1}}{(b - 1)!} \frac{1}{x - x_0} + \frac{(x_0 - {\rm i}\epsilon)^{b}}{b!} \frac{1}{(x - x_0)^2} - \sum_{a = 0}^{b - 1} \sum_{c = -(a + 2)}^{b - (a + 2)} \frac{(a + 1)x^{c}(-{\rm i}\epsilon)^{b - (c + a + 2)}x_0^{a}}{(c + a + 2)!(b - c - a - 2)!} \nonumber\\
\quad\quad{}+ \sum_{a_1 + a_2 + a_3 + a_4 = b - 2} \frac{(a_1 + a_2 + 1)!}{a_1!a_2!} \frac{(a_3 + a_4)!}{a_3!a_4!} x_0^{a_1} (-{\rm i}\epsilon)^{a_2 + a_3} x^{a_4}. \label{RDED}
\end{gather}
Notice that
\begin{gather*}
\sum_{a = 0}^{b - 1} \sum_{c = -(a + 2)}^{b - (a + 2)} = \sum_{a = 0}^{b - 1} \sum_{\tilde{c} = 1}^{a + 2} + \sum_{a = 0}^{b - 2} \sum_{c = 0}^{b - (a + 2)} = \sum_{\tilde{c} = 2}^{b + 1} \sum_{a = \tilde{c} - 2}^{b - 1} + \sum_{a = 0}^{b - 1} \delta_{\tilde{c},1} + \sum_{c = 0}^{b - 2} \sum_{a = 0}^{b - (c + 2)},
\end{gather*}
where $\tilde{c} = -c$. The sum of the $\tilde{c} = 1$ terms is computed by a binomial formula. Since from the beginning the left-hand side of \eqref{RDED} quantity should be $O\big(1/x^2\big)$ when $x \rightarrow \infty$, we must have cancellations of the polynomial part in $x$. After index relabelings we find
\begin{gather*}
 \frac{(x - {\rm i}\epsilon)^{b}}{b!} \frac{\check{\mathcal{R}}_{b}[f_{x_0}](x)}{\dd x}
 = \frac{(x_0 - {\rm i}\epsilon)^{b - 1}}{(b - 1)!}\left(\frac{1}{x - x_0} - \frac{1}{x}\right) \\
 \qquad{} + \frac{(x_0 - {\rm i}\epsilon)^{b}}{b!} \frac{1}{(x - x_0)^2} - \sum_{c = 1}^{b} \sum_{a = 0}^{b - c} \frac{(a + c)(-{\rm i}\epsilon)^{b - a}x_0^{a + c - 1}}{a!(b - a)!} \frac{1}{x^{c + 1}}.
\end{gather*}
We deduce when $\epsilon \rightarrow 0$
\begin{gather*}
 2\operatorname{Re}\bigg(\int_{{\rm i}\epsilon}^{{\rm i}\infty} \dd x\,\frac{(x - {\rm i}\epsilon)^{b}}{b!}\,\check{\mathcal{R}}_{b}[f_{x_0}](x)\bigg)
 = 2\operatorname{Re} \bigg( \frac{(x_0 - {\rm i}\epsilon)^{b - 1}}{(b - 1)!}\ln\left(\frac{{\rm i}\epsilon}{{\rm i}\epsilon - x_0}\right) \\
 \qquad\quad{} + \frac{(x_0 - {\rm i}\epsilon)^{b}}{b!}\frac{1}{{\rm i}\epsilon - x_0} - \sum_{c = 1}^{b} \sum_{a = 0}^{b - c} \frac{(a + c)(-1)^{b - a}({\rm i}\epsilon)^{b - a - c}x_0^{a + c - 1}}{a!(b - a)! c}\bigg) \\
\qquad{} = \frac{2x_0^{b - 1}}{(b - 1)!}\ln(\epsilon) - \frac{x_0^{b - 1}}{(b - 1)!} \ln\big(x_0^2\big) - \frac{2x_0^{b - 1}}{b!} - 2\left(\sum_{c = 1}^b \frac{(-1)^{c}}{(b - c)!c!\,c}\right) bx_0^{b - 1} + o(1).
\end{gather*}
Also
\begin{align*}
\sum_{c = 1}^b \frac{(-1)^{c}}{(b - c)!c!\,c} & = \frac{1}{b!} \int_{0}^{1} \frac{(1 - x)^{b} - 1}{x}\,\dd x
 = \frac{1}{b!}\lim_{\epsilon \rightarrow 0^+} \left(\int_{0}^{1} (1 - x)^{b}x^{-1 + \epsilon}\,\dd x - \frac{1}{\epsilon}\right) \\
& = \lim_{\epsilon \rightarrow 0^+} \left(\frac{\Gamma(\epsilon)}{\Gamma(\epsilon + b + 1)} - \frac{1}{b! \epsilon}\right)
 = -\frac{H_{b}}{b!},
\end{align*}
where we have used that $\Gamma(\epsilon) = 1/\epsilon - \gamma_{E} + o(1)$ when $\epsilon \rightarrow 0$. We arrive at
\begin{align*}
\mathcal{I}_{b}^0[f_{x_0}] &= \frac{x_0^{b - 1}}{(b - 1)!}\,\ln\big(x_0^2\big) + \frac{2x_0^{b - 1}}{b!} - \frac{2x_0^{b - 1}\,H_b}{(b - 1)!} \\
& = \frac{x_0^{b - 1}}{(b - 1)!}\,\ln\big(x_0^2\big) + \frac{2x_0^{b - 1}}{b!}\,(1 - bH_{b}).\tag*{\qed}
\end{align*}\renewcommand{\qed}{}
\end{proof}

\subsection*{Acknowledgements}

This work was initiated during a visit of G.B.~at the University of Melbourne supported by P.~Zinn-Justin, which he thanks for hospitality. G.B.~also thanks Hiroshi Iritani for discussions on mirror symmetry, and acknowledges the support of the Max-Planck-Gesellschaft. Part of this work was carried out during a visit of P.N.~to Ludwig-Maximilians-Universit\"at which he thanks for its hospitality. P.N.~is supported by the Australian Research Council grants DP170102028 and DP180103891.

\pdfbookmark[1]{References}{ref}
\LastPageEnding

\end{document}